\documentclass[a4paper,12pt,oneside,reqno]{amsart}

\usepackage{hyperref}
\usepackage[headinclude,DIV13]{typearea}
\areaset{15.1cm}{25.0cm}
\usepackage{txfonts,amssymb,amsmath,amsthm,bbm}
\usepackage{footmisc}
\usepackage[latin1] {inputenc}
\usepackage{xcolor}
\usepackage{subfigure}
\usepackage{comment}
  
\newtheorem{theorem}{Theorem}[section]
\newtheorem{lemma}[theorem]{Lemma}
\newtheorem{proposition}[theorem]{Proposition}
\newtheorem{corollary}[theorem]{Corollary}
\newtheorem{remark}[theorem]{Remark}

\definecolor{bbm}{RGB}{51,153,0}
\definecolor{above}{RGB}{128,0,128}
\definecolor{below}{RGB}{102,0,204}
\definecolor{cascade}{RGB}{204,0,0}
\definecolor{iid}{RGB}{153,51,0}

\def\paragraph#1{\noindent \textbf{#1}}

\numberwithin{equation}{section}

\def\Var{\mathop{\rm Var}\nolimits}
\def\Cov{\mathop{\rm Cov}\nolimits}

\def\<{\langle}
\def\>{\rangle}

 \def \ba {\begin{array}}
 \def \ea {\end{array}}

 \newcommand{\be}{\begin{equation}}
 \newcommand{\ee}{\end{equation}}

\newcommand{\bea}{\begin{eqnarray}}
 \newcommand{\eea}{\end{eqnarray}}
\def\TH(#1){\label{#1}}\def\thv(#1){\ref{#1}}
\def\Eq(#1){\label{#1}}\def\eqv(#1){(\ref{#1})}

 \def \1{\mathbbm{1}}
\def\wt {\widetilde}
\def\wh{\widehat}

\newcommand{\bbG}{\mathbb{G}}

\newcommand{\bbH}{\mathbb{H}}
\newcommand{\bbL}{\mathbb{L}}
\newcommand{\bbT}{\mathbb{T}}

\newcommand{\bbE}{\mathbb{E}}
\newcommand{\bbW}{\mathbb{W}}

\newcommand{\bbV}{\mathbb{V}}
\newcommand{\bbU}{\mathbb{U}}

\newcommand{\bbN}{\mathbb{N}}
\newcommand{\bbZ}{\mathbb{Z}}
\newcommand{\bbR}{\mathbb{R}}

\newcommand{\ol}{\overline}

\newcommand{\ul}{\underline}

\newcommand{\cI}{\mathcal I}

\newcommand{\cG}{\mathcal G}

\newcommand{\cE}{\mathcal E}
\newcommand{\cN}{\mathcal N}
\newcommand{\cF}{\mathcal F}

\newcommand{\rmc}{{\rm c}}

\newcommand{\rmd}{{\rm d}}
\newcommand{\rme}{{\rm e}}

\newcommand{\rmP}{{\rm P}}
\newcommand{\rmQ}{{\rm Q}}

\newcommand{\rmE}{{\rm E}}
\newcommand{\rmVar}{{\rm Var}}
\newcommand{\rmCov}{{\rm Cov}}

\newcommand{\sgn}{\mathop {\rm sgn}}

\newcommand{\eqd}{\overset{\rm d}{=}}

\usepackage{xr}
\begin{document}


 \title[DGFF Subject to a hard wall / bounds]
{Gaussian Free Field on the Tree Subject to a Hard Wall \\I: BOUNDS}
\author[M. Fels]{Maximilian Fels}
\address{M. Fels, 
	Technion - Israel Institute of Technology. Haifa, 3200003,
	Israel.}
\email{felsm@campus.technion.ac.il}
\author[L. Hartung]{Lisa Hartung}
\address{L. Hartung, Universit\"at Mainz. 
	Staudingerweg 9,
	55128 Mainz, Germany.}
\email{lhartung@uni-mainz.de}
\author[O. Louidor]{Oren Louidor}
\address{O. Louidor,
	Technion - Israel Institute of Technology. Haifa, 3200003,
	Israel.}
\email{oren.louidor@gmail.com}

\date{\today}

\begin{abstract}
This is the first in a series of two works which study the discrete Gaussian free field on the binary tree when all leaves are conditioned to be positive. In this work, we obtain sharp asymptotics for the probability of this ``hard-wall constraint'' event, and identify the repulsion profile followed by the field in order to achieve it. We also provide estimates for the mean, fluctuations and covariances of the field under the conditioning, which show that in the first log-many generations the field is localized around its mean. These results are used in the sequel work~\cite{Work2} to obtain a comprehensive asymptotic description of the law of the field under the conditioning.
\end{abstract}

\maketitle

\tableofcontents

\section{Introduction and results}
\subsection{Setup, motivation and previous results}
The discrete Gaussian free field (DGFF) on the infinite binary tree $\bbT$ rooted at $0$ (with Dirichlet boundary conditions) is a centered Gaussian process $h = (h(x) :\: x \in \bbT)$ with covariances given by
\begin{equation}
\label{e:01.1}
\rmE\big[h(x) h(y)\big] = \frac12 \big(|x| + |y|) - \rmd_{\bbT}(x,y)\big) = |x \wedge y|\,.
\end{equation}
Above $\rmd_{\bbT}$ is the graph distance on the tree, $|x| := \rmd_{\bbT}(x,0)$ is the depth of $x \in \bbT$ and $x \wedge y$ denotes the deepest common ancestor of $x,y \in \bbT$. We shall often also write $[x]_k$ for the ancestor of $x$ in generation $k \leq |x|$.
Alternatively $h$ can be seen as a branching random walk (BRW) with fixed binary branching and standard Gaussian steps, in which case we shall use the term generation instead of depth. We shall refer to $h$ as both a DGFF and a BRW interchangeably.

Denoting by $\bbT_n$, resp. $\bbL_n$, the sub-graph of $\bbT$ which includes all vertices at depth at most, resp. equal,  to $n \geq 0$, the goal of this manuscript and the follow-up work~\cite{Work2} is to study the realization by the field of the event
\begin{equation}
\label{e:1}
\Omega_n^+ := \big\{ h(x) \geq 0 :\: x \in \bbL_n \big\} \,.
\end{equation}
The motivation for studying this event, comes from the area of random surfaces. Indeed, $h$ is the $\bbV = \bbT$, $\bbU = \{0\}$ case of the general class of discrete Gaussian free fields on (finite, non-empty) graph $\bbV$ with Dirichlet boundary conditions on $\bbU \subset \bbV$, whose laws can be expressed in terms of the Gibbs formulation:
\begin{equation}
\label{e:1.3a}
\begin{split}
	\rmP_{\bbV, \bbU}(\rmd h) = \frac{1}{Z_{\bbV, \bbU}} \exp \bigg(-\frac12 \big\|\nabla h\big\|_2^2 \bigg) \, 
	\prod_{x \in \bbU} \delta_0 \big(\rmd h(x)\big)\, \prod_{x \in \bbV \setminus \bbU} \rmd h(x) \,.
\end{split}
\end{equation}
Above $\nabla h(x) = \big(h(y) - h(x) :\: y \sim x\big)$, where $\sim$ is the neighboring relation on $\bbV$ and $Z_{\bbV, \bbU}$ is the normalizing constant.

When $\bbV \subset \bbZ^d$, the value of $h(x)$ can be interpreted as the height at $x \in \bbV$ 
of a random discrete surface or interface which lies above/below a domain $\bbV$. In this context, the event
\begin{equation}
	\Omega_{\bbW}^+ = \big\{h(x) \geq 0 :\: x \in \bbW \big\} \,,
\end{equation}
can be seen as reflecting the existence of a hard wall, which prevents the interface from penetrating the wall placed at the domain $\bbW$. The effect of this {\em hard wall constraint} constitutes a major research direction in the study of Gibbs-gradient fields (namely, random interfaces whose law is governed by their gradient; also known as Ginzburg-Landau fields), of which the discrete Gaussian free field is a canonical and prominent member (see, e.g., \cite{dembo2005survey, giacomin2001survey, bolthausen2002survey, velenik2006survey} for excellent surveys on the subject).

The most challenging, yet arguably physical, dimension is two.
This was treated in the seminal work of Bolthausen, Deuschel and Giacomin~\cite{bolthausen2001entropic}, who considered the DGFF on 
$\ol{\bbV_N}$ with zero boundary conditions at $\partial \bbV_N$, where
$\bbV_N := N V \cap \bbZ^2$ for $V = [0,1]^2 \subset \bbR^2$ and $N \in \bbN$. Here and after we shall denote by $\partial \bbU$ the outer boundary of $\bbU$ in $\bbZ^d$ and $\ol{\bbU} := \bbU \cup \partial \bbU$. Taking an open $W \subset V$ with smooth boundary and (to avoid boundary effects of a different probabilistic nature) positive distance to $\partial V$ and letting $\bbW_N = N W \cap \bbZ^2$, they showed 
\begin{equation}
\label{e:01.5}
\frac{\log \rmP_{\ol{\bbV}_N, \,\partial \bbV_N} \big(\Omega_{\bbW_N}^+)}{(\log N)^2} \underset{N \to \infty} \longrightarrow -4g\, {\rm cap}_V(W) \,,
\end{equation}
where $g=2/\pi$ and ${\rm cap}_V(W)$ is the relative capacity of $W$ with respect to $V$, namely
\begin{equation}
\label{e:101.6}
	\inf \Big\{\tfrac12 \|\nabla f\|_2^2 :\: f \in \bbH^{1}_0(V)\,,\,\, f_W \geq 1 \Big\}\,.
\end{equation}
Above $\bbH^1_0$ is the Sobolev space of functions with a square integrable weak derivative, vanishing at $\partial V$, and here and after $\varphi_A$ denotes the restriction of a function $\varphi$ to a subset $A$ of its domain.

The paper also shows that
\begin{equation}
\label{e:01.7}
	\sup_{x \in \bbW_N} \rmP \Big(\big|h(x) - 2\sqrt{g} \log N \big| > \epsilon \log N \,\Big|\, \Omega_{\bbW_N}^+ \Big) \underset{N \to \infty}\longrightarrow 0 \,.
\end{equation}
This shows that under the positivity conditioning, the field does not only stay above the wall but is, in fact, repelled away from it by $2\sqrt{g} \log N$ on first order. As the authors explain (and the proof shows), this repulsion is an entropy effect. Under the constraint, the field ``makes room'' for its global minima to stay positive, which unconditionally obeys
\begin{equation}
\label{e:01.8}
\frac{-\min_{x \in \bbW_N} h(x)}{\log N} \underset{N \to \infty}{\overset{\rmP_{\ol{\bbV}_N, \,\partial \bbV_N}}\longrightarrow} 2 \sqrt{g} \,.
\end{equation}
Above the limiting arrow denotes convergence in probability.

While (as far as the authors are aware of) the above first order results remain the state-of-the art concerning the hard-wall problem in two dimensions, considerable progress has been made in the theory of extreme order statistics for the field in this case. In particular, a more recent seminal work of Bramson, Ding and Zeitouni~\cite{BraDiZei16} improved upon~\eqref{e:01.8} by showing
\begin{equation}
\label{e:01.9}
-\min_{x \in \bbV_N} h(x) - m^{(\bbZ^2)}_N \underset{N \to \infty}{\overset{\rmd}\longrightarrow	}
	G \oplus \tfrac{\sqrt{g}}{2} \log Z^{(\bbZ^2)}
\end{equation}
where
\begin{equation}
\label{e:01.10}
m^{(\bbZ^2)}_N := 2\sqrt{g} \log N- \tfrac{3}{4} \sqrt{g} \log \log N \,,
\end{equation}
$G$ has a Gumbel law with rate $2/\sqrt{g}$ and $Z^{(\bbZ^2)} \in (0,\infty)$ is a non-trivial random variable, which is distributed as the total mass of the so-called Liouville quantum gravity measure associated with the continuous Gaussian free field on $[0,1]^2$. Above $\oplus$ stands for the sum of two independent random variables.
The asymptotic joint law of all min-extremes of the field, which is commonly recorded via the so-called (full) extremal process of the field, was derived by Biskup and the last author in~\cite{BisLou16, BisLou18, BisLou20}.

As the extremal behavior of the DGFF in 2D is now essentially fully understood, one hopes that sharper results beyond those in~\eqref{e:01.5},~\eqref{e:01.7} for the hard wall problem could be obtained. Indeed, as was argued above, the statistics of extremes constitute a determining factor in the realization of this event. This paper can be thought of as a step in this direction, albeit for the tree graph instead of a dimensional domain.

While seemingly very different at first sight, the DGFF on two dimensional domains and that on the tree graph $\bbT$ bare substantial similarity. As both are centered Gaussian fields, this can be seen by comparing covariances, which in the 2D case are given by 
\begin{equation}
	\rmE_{\bbV_N, \partial \bbV_N} h(x) h(y) = G_{\bbV_N, \partial \bbV_N}(x,y) 
	= g \big(\log N - \log \|x-y\|_2\big) + O(1)
	\ , \quad x,y \in \bbV_N \,.
\end{equation}
Above $G_{\bbV_N, \partial \bbV_N}$ is the (negative) discrete Green function on $\bbV_N$ (w.r.t. the simple random walk on $\bbZ^2$), and the second equality holding ``in the bulk'' of $\bbV_N$.
Indeed, for $x,y \in \bbL_n$, the covariance in~\eqref{e:01.1} compares to the above if one replaces $n$ by $\log_2 N$ (for $N=2^n$) and $\rmd_{\bbT}$ by the logarithm of the Euclidean distance. To see that the latter are analogous, one may embed dyadically the vertices of $\bbL_n$ in the interval $[0,N) \cap \bbZ$. Then for ``typical'' pairs $x,y \in \bbL_n$, 
\begin{equation}
	\frac12 \rmd_{\bbT}(x,y) \cong \log_2 \|{\rm x}-{\rm y}\|_2 \,,
\end{equation}
with ${\rm x}$, ${\rm y}$ on the right being the corresponding points of $x$,$y$ in $[0,1]$ under the embedding.
The DGFF on planar domains and on the tree are thus both instances of a large class of fields with correlations which decay logarithmically, or approximately logarithmically, with the distance between test points, under some intrinsic metric on the domain. There is a growing body of works showing the universality of extreme value behavior for such fields. See, e.g.~\cite{DRZ17, LP19, SZ24}.

The advantage of working with the tree is that the geometry generated by the graph distance is much simpler to analyze than that of the Euclidean one, while at the same time, the extreme value theory for this graph is equally advanced. In particular, preceding the work~\cite{BraDiZei16}, the following tree analog of~\eqref{e:01.9} was shown by Aidekon~\cite{AidekonBRW13} in another seminal paper in this area:
\begin{equation}
\label{e:01.13}
-\min_{x \in \bbL_n} h(x) - m_n
\underset{n \to \infty}{\overset{\rmd} \longrightarrow}
	G \oplus c_0^{-1} \log Z \,,
\end{equation}
where
\begin{align}\label{Lisa.1}
m_n=c_0 n- \tfrac{3}{2}c_0^{-1}\log n \quad ; \qquad c_0=\sqrt{2\log 2} \,,
\end{align}
$G$ has a Gumbel law with rate $c_0$ and $Z \in (0,\infty)$ is now the limit of the so-called derivative martingale associate with $h$.

Asymptotics for the probability of the event $\Omega_n^+$ on the tree was treated recently by both Chen and He~\cite{Chen2018} and Roy~\cite{Roy_2024}. In~\cite{Chen2018}, the authors study the lower deviation events 
\begin{equation}
\label{e:01.14}
	\Omega_n(u) := \Big\{	\min_{\bbL_n} h \geq -m_n + u \Big\}
	\quad ; \qquad u \in \bbR
\end{equation}
at $u = O(n)$, for various types of BRWs, including that in the present manuscript. Their work shows
\begin{equation}
\label{e:01.16}
	-\log \rmP(\Omega_n(u)) = \frac12 u^2 \,(1+o(1)) \,,
\end{equation}
for $u = O(n)$ but tending to $\infty$ with $n$. In particular, this gives
$-\frac12 m_n^2 \,(1+o(1))$ as the asymptotics for the logarithm of the probability of $\Omega_n^+$, which is the tree version of~\eqref{e:01.5}.

A sharper version,
\begin{equation}
\label{e:01.17}
	-\log \rmP(\Omega_{n}^+) = \frac12 
	\big(m_n - c_0 \log_2 n\big)^2 + \Theta(n) \,,
\end{equation}
with $\Theta(n)/n$ - a bounded sequence in $n$, was given in~\cite{Roy_2024}, in which it is also shown that
\begin{equation}
\label{e:01.18}
m_n - 3c_0 \log n \leq 
	\rmE \big( h(x) \,\big|\, \Omega_n^+\big) \leq m_n - c_0 \log n\,,
\end{equation}
for $x \in \bbL_n$.
The second result shows, in analog to~\eqref{e:01.7}, that a similar entropic repulsion phenomenon occurs also in the case of the tree. However, it also already shows that the repulsion level is strictly smaller (with the difference tending to $\infty$ with $n$) than the height of the unconstrained minimum. 

\subsection{Results}
This manuscript is the first in a series of two works, in which we aim to improve upon the above results, by providing a rather precise description of the field $h$ subject to the hard wall constraint $\Omega_n^+$. This program begins in the present work with the derivation of fairly strong versions of~\eqref{e:01.16} and~\eqref{e:01.17}. These are in turn used to obtain sharp bounds on the mean, covariances and tails of the field under $\Omega_n^+$, in particular improving upon~\eqref{e:01.18} by ``nailing'' the repulsion up to $O(1)$. Building on these results, the second work~\cite{Work2} provides a relatively broad asymptotic picture of the statistics of the field under the conditioning on $\Omega_n^+$, both locally, namely the conditional law of the field in a neighborhood of a vertex, and globally, namely the law of global observables such as the field global maximum or minimum.

Turning to the statements themselves, the first result gives sharp asymptotics for the right tail event $\Omega_n(u)$ of the minimum. For what follows, for $s > 0$, we let
\begin{equation}
\label{e:1.7}
[s]_2 := \log_2 s - \lfloor \log_2 s \rfloor \,,
\end{equation}
and set $[s]_2 \equiv 0$ if $s \leq 0$.
\begin{theorem}\label{t:1.1}
There exists a bounded function $\theta: [0,1) \to \bbR$ such that for $u \in \bbR$, $n \geq 1$,
\begin{equation}
\label{e:1.3}
-\log \rmP \big(\Omega_n(u)) = 
	\frac{\big(u^+- c_0 \log_2 (u \vee 1)\big)^2}{2} + \theta_{[u]_2}u + o(u^+) \,,
\end{equation}
where $o(u)/u \to 0$ as $u \to \infty$ uniformly in $n$ such that $u \leq 2^{\sqrt{n}}$.
Moreover, 
\begin{equation}
-\frac{\rmd}{\rmd u} \log \rmP \big(\Omega_n(u)) = u^+- c_0 \log_2 (u \vee 1) + O(1)\,,
\end{equation}
where the $O(1)$ term is bounded uniformly in $n$ and $u$ such that $u \leq 2^{\sqrt{n}}$.
\end{theorem}

\subsubsection{Remarks on Theorem~\ref{t:1.1}}
Let us make several remarks concerning the theroem.
\begin{enumerate}
\item 
The two parts of the theorem offer different formulations for the asymptotics of the right tail event, which do not follow from one another. While the first one is perhaps more natural, it is the second one, where the error is expressed in terms of the derivative, that constitutes a crucial input, both for controlling the covariances (Theorem~\ref{t:1.4}) and for the derivation of the asymptotic law of the field under the conditioning~\cite{Work2} (see Subsection~\ref{s:1.3} here and Subsection~\ref{2@s:1.4} in~\cite{Work2}).

\item It is not clear whether the dependence on $[u]_2$ in the constant behind the linear term in~\eqref{e:1.3} is indeed there. This dependence arises as a result of the proof which involves studying the field in the first $\lfloor \log_2 u \rfloor = \log_2 u - [u]_2$ generations. It is therefore an artifact of the discreteness of the domain (or time evolution) of our model. Nevertheless, as $\theta_{[u]_2}$ is a sum of terms, some of which are only defined implicitly in the proof, we could not rule out (or in) that this is a fixed constant. The question of whether the discreteness of the model survives this way in the limit of the right tail of the centered minimum is thus left open. We note that thanks to the continuity of the limiting law of the right tail (see, e.g., Lemma~\ref{l:2.9a}), $\theta_{[u]_2}$ is continuous in $u$.

\item The restriction on $u$ in relation to $n$ is not sharp and the argument can be easily pushed to show that the asymptotics in both statements hold uniformly whenever $u \leq 2^{n^{1-\epsilon}}$ for any $\epsilon > 0$. Nevertheless, as we are after the case $u = m_n$, we did not attempt to give the most general statement. We note, however, that for $u$ which is of the order $2^n$ or larger, we do not expect such asymptotics to hold, as the field does not have the $\lfloor \log_2 u \rfloor$ generations required to optimally rise to level $u + \Theta(1)$ (see Subsection~\ref{s:1.3}). 	
\item The asymptotics of the right tail event, and consequently the results that follow, would have been very different had we allowed random branching. Indeed, then a possible way of increasing the minimal height of all particles is by suppressing all branching until a late time. As our motivation comes from the DGFF point-of-view of this model, we did not study such case. In the continuum, the case of random branching was considered, for example, in~\cite{CHM23,BaiH22}, where the underlying process is branching Brownian motion.

\end{enumerate}
As an immediate consequence of the theorem we get,
\begin{corollary}\label{cor:0.4}
With $\theta$ as in Theorem~\ref{t:1.1},
\begin{equation}\label{eq:0.23}
-\log \rmP \big( \Omega_n^+ \big) = \frac12 (m_n - c_0 \log_2 n)^2 + \theta_{[n]_2}n + o(n) \,.
\end{equation}
\end{corollary}

We now turn to study the field under the conditioning $\Omega_n^+$. For what follows we denote the law of $h$ restricted to $\bbT_n$ by $\rmP_n$ and set
\begin{equation}
	\rmP_n^+(\cdot) := \rmP_n(\cdot\,|\, \Omega_n^+) \,,
\end{equation}
with $\rmE_n$ and $\rmE_n^+$ defined accordingly. We also set 
\begin{equation}
l_n :=  \lfloor \log_2 n\rfloor \ \ , \ \ \ \ n' := n - l_n \,,
\end{equation}
so that 
\begin{equation}
	m_{n'} = m_n - c_0 \log_2 n + O(1) \,,
\end{equation}
and finally,
\begin{equation}
\label{e:1.10a}
	\mu_n(x) :=  m_{n'} \big(1 - 2^{-|x|}\1_{\{|x| < l_n\}}\big) 
	\quad ; \qquad x \in \bbT_n \,.
\end{equation}
The next result describes the repulsion profile of the field under $\rmP_n^+$. 
\begin{theorem}\label{t:1.3}
There exists $C,c \in (0, \infty)$ such that 
for all $n \geq 1$, $x \in \bbT_n$ and $u > 0$,
\begin{equation}
\label{e:1.9a}
c \exp \Big(-C \frac{u^2}{	\ol{\sigma}_n(x, u)} \Big)
\leq 
	\rmP^+_n \big(h(x) - \mu_n(x) > u \big) \leq C \exp \Big(-c \frac{u^2}{	\ol{\sigma}_n(x, u)} \Big)
\end{equation}
and 
\begin{equation}
\label{e:1.9b}
c \exp \Big(-C \frac{u^2}{	\ul{\sigma}_n(x)} \Big) 
\leq 
	\rmP^+_n \big(h(x) - \mu_n(x) < -u \big) \leq C \exp \Big(-c \frac{u^2}{	\ul{\sigma}_n(x)} \Big) \,,
\end{equation}
where 
\begin{equation}
\label{e:1.11a}
	\ul{\sigma}_n(x) := (|x| - l_n)^++1
	\  , \quad
	\ol{\sigma}_n(x, u) := \big(\log_2 (u \wedge n) - (l_n - |x|)^+\big)^+ + (|x|-l_n)^+ + 1 \,.
\end{equation}
If $|x| > l_n$, then the lower bound in~\eqref{e:1.9b} holds only up to $C(m_{|x|-l_n}+1)$.
\end{theorem}
While not entirely immediate, it does not take much to conclude from the theorem that,
\begin{corollary}
\label{c:1.4}
For all $n \geq 1$ and $x \in \bbT_n$ 
\begin{equation}\label{eq:1.14}
		\rmE_n^+ h(x) = m_{n'} \big(1 - 2^{-|x|}\1_{\{|x| < l_n\}}\big)  + O(1) \,,
\end{equation}
where $\mu_n(x)$ is as in~\eqref{e:1.10a} and $O(1)$ is a bounded sequence in $n$.
\end{corollary}

Thus, under the conditioning on $\Omega_n^+$ the field lifts to height $m_n'$ exponentially fast in the first $l_n = \lfloor \log_2 n \rfloor + O(1)$ generations, its height at depth $k \leq l_n$ strongly concentrated around $m_n'(1-2^{-k})$ with a Gaussian lower tail and an ``almost'' Gaussian upper tail: $\rme^{-\Theta(u^2/(\log u - (l_n-k))^+)}$ (both with constant coefficient). In particular 
for $x \in \bbT_n$ with $|x| = l_n + O(1)$ and $u = O(n)$, one has 
\begin{equation}
\label{e:01.31}
	\rmP^+_n \big(h(x) - m_{n'} > u \big) = \rme^{-\Theta(u^2/\log u)}
	\quad, \qquad
	\rmP^+_n \big(h(x) - m_{n'} < -u \big) = \rme^{-\Theta(u^2)} \,.
\end{equation}

Past depth $l_n$, the fluctuations in the height have Gaussian tails with variance which grow linearly in the depth, but the mean stays $m_n' + O(1)$ (the precise asymptotic law at all such vertices is given in the sequel work~\cite{Work2}). The (mean) repulsion level ``in the bulk'' is therefore $m_n' + O(1)$ which is $c_0 \log_2 n + O(1)$ below the typical (negative) height of the field's unconstrained global minimum. 

For the variance of the height of the field as well as the joint law of the heights at different vertices of $h$ under $\rmP_n^+$ we provide the following estimate for the conditional covariances.
\begin{theorem}\label{t:1.4}
There exists $c > 0$ such that for all $n \geq 1$ and $x,y \in \bbT_n$ with $|x|,|y| \geq l_n$, 
\begin{equation}
\label{e:1.14}
\rmCov_n^+\, \big(h(x), h(y) \big) 
= \left\{ \begin{array}{lll}
|x \wedge y| - l_n + O(1)
& \quad & |x \wedge y| \geq l_n \,,\\
O\big( \rme^{-c(l_n - |x \wedge y|)}\big)
& \quad & |x \wedge y| < l_n \,.
\end{array} \right.
\end{equation}
\end{theorem}
In particular, the variance of the height at a leaf in $\mathbb{L}_n$ under the conditioning is $n' + O(1)$ and leaves whose most recent ancestor is at generation $l_n$ or before, have approximately independent heights. This is already in stark contrast compared to the case of the unconditional field.

Combining~\eqref{e:01.31} and~\eqref{e:1.14} with $x,y \in \bbL_{l_n}$, it follows that the heights of the conditional field at vertices in generation $l_n$ are essentially i.i.d. random variables which are tight around $m_{n'}$. Using the Markov property of $h$, which survives under $\rmP_n^+$, we thus see that beyond depth $l_n$, the conditional field has essentially the same law as $|\bbL_{l_n}| = 2^{l_n}=n$ independent DGFFs on a tree of depth $n-l_n = n'$, each starting at a random height, which is tight around $m_{n'}$, and conditioned to stay above $0$. Since $m_{n'}$ is the typical (negative) height of the global minimum of each of these DGFFs, this conditioning remains non-singular in the limit.

In the sequel paper, we show that under $\rmP_n^+$ the law of $h(x)-m_{n'}$ for $x \in \bbL_{l_n}$ is not just tight in $n$, but also tends to a limit as $n \to \infty$. We then use the above description to obtain convergence in law for the height of the field at all $x \in \bbT_n$ with $|x| \geq l_n$, in both a local and global sense.

\subsection{Proof overview}
\label{s:1.3}
Let us now give an overview of the proofs of the main results in this paper, starting with the proof of Theorem~\ref{t:1.1}
\subsubsection{Sharp control of the right tail of the minimum}
It is clear that in order to achieve the event $\Omega_n(u) = \{\min_{\bbL_n} h \geq -m_n + u\}$,  namely that the global minimum increases by some $u \gg 1$ above its typical value, the field has to rise to height $u + \Theta(1)$ at some generation $k$, so that the minimal height at the leaves of all subtrees rooted in that generation increase accordingly. Otherwise, there will be an exponential cost in the number of $x \in \bbL_k$ such that $h(x) < u + \Theta(1)$.

Now, on one hand $k$ should be as small as possible, so that less generations need to be controlled. On the other hand, $k$ should be as large as possible, so that the increase in the field's height is taken over more steps. The exponential (probabilistic) cost of $\{h(x) \geq v :\: x \in \bbL_k\}$ on first order can be determined by minimizing the Dirichlet Energy term in the exponential in~\eqref{e:1.3a}, namely
\begin{equation}
	\inf_{f \in \bbR^{\bbT_k}}  \Big\{\tfrac12 \big\|\nabla f\big\|_2^2 :\: 
f(0) = 0,\, f_{\bbL_k} = v \Big\} 
= v^2 \inf_{f \in \bbR^{\bbT_k}}  \Big\{\tfrac12 \big\|\nabla f\big\|_2^2 :\: 
f(0) = 0,\, f_{\bbL_k} = 1 \Big\} \,,
\end{equation}
with the second minimization being the discrete analog of the relative capacity from~\eqref{e:101.6}.

The infima above are achieved by the 
discrete harmonic extensions on $\bbT_k$ with the prescribed boundary values. That is, focusing on the first infimum, by the function $\mu^{(v,k)} :\: \bbT_k \to \bbR$ which satisfies
\begin{equation}
	\Delta \mu^{(v,k)} = 0 \quad \text{on} \quad \bbT_k \setminus (\{0\} \cup \bbL_k) \quad, \qquad
	\mu^{(v,k)}(0) = 0\,,\,\, \mu^{(v,k)}_{\bbL_k} = v \,,
\end{equation}
where $\Delta \mu^{(v,k)}(x) = -\sum_{y \sim x} \big(\mu^{(v,k)}(y) - \mu^{(v,k)}(x)\big)$ is the (discrete, negative) Laplacian on $\bbT_k$ (see Subsection~\ref{s:2.2}). It is not difficult to show using random walk estimates that (see Subsection~\ref{s:2.2} again),
\begin{equation}
\frac12 \big\|\nabla \mu^{(v,k)}\big\|_2^2 = 
\frac12 \big\langle \mu^{(v,k)}, \Delta \mu^{(v,k)}\big\rangle = \frac12 v^2 \Big(1+\frac{1}{2^k-1}\Big) 
\quad, \qquad
	\mu^{(v,k)}(x) = \frac{1-2^{-|x|}}{1-2^{-k}} v \,.
\end{equation}
Plugging in $v=u + \Theta(1)$ and setting 
\begin{equation}
l_u := \lfloor \log_2 u \rfloor \,,
\end{equation}
we see that the first order exponential cost is $\gg \frac12 u^2 + \Theta(u)$ if $k \ll l_u$ and equal to $u^2 + \Theta(u)$ otherwise. 
This shows that the minimal $k$ which achieves optimal first order cost is $k = l_u + \Theta(1)$, and suggests that the optimal way to realize the event $\Omega_n(u)$ is by having the field reach height $u + \Theta(1)$ at all vertices in that generation.

A key observation now is that at generation $l_u + \Theta(1)$, the remaining depth of the tree has already decreased from $n$ to $n-l_u$. This implies that the typical minimal height at the leaves of each subtree rooted in this generation has already increased from $-m_n$ to $-m_{n-l_u}$, i.e. by
\begin{equation}
	m_n - m_{n-l_u} = c_0 l_u + O(1) 
\end{equation}
(assuming $u = O(n)$). This implies that the field, in fact, need only reach height $u' + \Theta(1)$ in generation $l_u + \Theta(1)$, where 
\begin{equation}
u' := u - m_n + m_{n-l_u} = u - c_0 l_u + O(1)\,.
\end{equation}
In particular, the revised first order exponential cost and optimal repulsion profile are
\begin{equation}
\frac12 \big\langle \mu^{(u', l_u)}, \Delta \mu^{(u', l_u)}\big\rangle =  \frac12 (u')^2 + \Theta(u) 
\quad, \qquad
	\mu^{(u', l_u)}(x) = \big(1-2^{-|x|}\big) \big(u' + \Theta(1)) \,.
\end{equation}

To make this heuristics rigorous, we condition on the value of $h$ at generation $l_u$ to write
\begin{equation}
\label{e:101.40}
\rmP \big(\Omega_n(u) \big) =  \rmE_{l_u} \varphi_{n,{l_u},u} (h) \,,
\end{equation}
where
\begin{equation}
\label{e:1001.40}
\varphi_{n,l_u,u} (h) = \prod_{x \in \bbL_{l_u}} 
p_{n-l_u}\big(u' - h(x) \big) \,,
\end{equation}
and
\begin{equation}
p_k(v) := \rmP_k\big(\Omega_k(v)\big) \,,
\end{equation}
are the tail probability functions for the law of the minimum.
We then {\em tilt} the measure $\rmP_{l_u}$ by $\mu^{(u',l_u)}$, via the Cameron-Martin Formula:
\begin{equation}
	\rmE_k F(h) = \rme^{- \tfrac12 \langle \mu, \Delta \mu \rangle} \rmE_k \Big( F(h + \mu) \rme^{- \langle h, \Delta \mu \rangle} \Big) \,,
\end{equation}
with the choices: $k = l_u$, $\mu = \mu^{(u',l_u)}$ and $F = \varphi_{n,{l_u},u}$ as above. This equates $\rmP_n(\Omega_n(u))$ as
\begin{equation}
\label{e:101.43}
	\rme^{-\frac{u'^2}{2} - C_{[u]_2} u + o(u)}\,
	\rmE_{l_u} 
	\exp 
	\sum_{x \in \bbL_{l_u}} g_{n,u}\big(h(x)\big) 
\quad ; \quad
g_{n,u}(v) := \log p_{n-l_u} (-v) -\big(C'_{[u]_2} + o(1)\big) v \,,
\end{equation}
where $C_\delta$, $C'_\delta$ for $\delta \in [0,1)$ are uniformly bounded constants.

Now thanks to the convergence of the centered minimum~\eqref{e:01.13} and the a-priori tail estimates thereof~\eqref{e:01.16}, the functions $g_{n,u}$ satisfy (uniformly in $u$):
\begin{description}
    \item[C1] $g_{n,u}(v) \to g_{\infty,u}(v)$ as $n \to \infty$ uniformly in $v$ on compact subsets of $\bbR$.
    \item[C2] $g_{\infty,u}$ is continuous.
    \item[C3] $\lim_{v \to \pm \infty} \sup_n g_{n,u}(v) = -\infty$.
\end{description}
Proposition~\ref{l:3.2} then shows that under such conditions, there exists $G_u^* = G_{[u]_2}^*$ such that 
\begin{equation}
\frac{1}{|\bbL_{l_u}|} \log \Big[ \rmE_{l_u} 
	\exp 
	\sum_{x \in \bbL_{l_u}} g_{n,u}\big(h(x)\big) \Big]
 = G^*_{[u]_2} + o(1) \,.
\end{equation}
Plugging this in~\eqref{e:101.43} and observing that $|\bbL_{l_u}| = 2^{-[u]_2} u$, this gives the first part of Theorem~\ref{t:1.1} with $\theta_\delta = -C_\delta+ 2^{-\delta} G^*_\delta$.

To control the derivative of $\rmP(\Omega_n(u))$, we first differentiate the right hand side in~\eqref{e:101.40} with the law $\rmP_{l_u}$ expressed in terms of its Gibbs formulation~\eqref{e:1.3a}. This gives
\begin{equation}
\label{e:101.45}
\frac{\rmd}{\rmd u} \big(-\log p_n(u)\big) = \frac{1}{1-2^{-l_u}} \rmE_n \big(h(x)  \,\big|\, 
		\Omega_n(u)\big) 
\quad ; \qquad x \in \bbL_{l_u}\,.
\end{equation}
It is thus sufficient to show that the above conditional mean is $u' + O(1)$. 

For a lower bound on this mean, we use~\eqref{e:101.40} again to write the conditional mean as
\begin{equation}
\label{e:101.46}
	\frac{\rmE_{l_u} \big(h(x) \varphi_{n,l_u,u}(h)\big)}
	{\rmE_{l_u} \varphi_{n,l_u,u}(h)} \,,
\end{equation}
with $\varphi_{n,l_u,u}$ as in~\eqref{e:1001.40}.
Using the Gaussian upper bound~\eqref{e:01.16} on the right tail of the minimum, one may show that $-\frac18(v-C)^2 - \log p_{k}(v)$ is increasing in $v$ for sufficiently large $C < \infty$, uniformly in $k$. It then follows from the FKG property of $h$ (see Subsection~\ref{s:2.4}), that replacing $\varphi_{n,l_u,u}$ in~\eqref{e:101.46} by
\begin{equation}
\psi_{n,l_u,u}(h) := \prod_{x \in \bbL_{l_u}} \tilde{p} \big(u'-h(x) \big) \quad ; \qquad 
\tilde{p}(v) := \rme^{-\frac18 (v-C)^2} \,,
\end{equation}
can only make the resulting quantity smaller. 

The quotient~\eqref{e:101.46} with $\psi$ in place of $\varphi_{n,l_u,u}$ 
can then be interpreted as the mean at $x \in \bbL_{l_u}$ of a DGFF on the graph $\bbV = \wt{\bbT}_{l_u+1}$, which is obtained from $\bbT_{l_u}$ by adding a new neighbor $y'$ to each $y \in \bbL_{l_u}$ and imposing boundary conditions $0$ at $0$ and $u'-C$ at $\wt{\bbL}_{l_u+1} = \{y' :\: y \in \bbL_{l_u}\}$. Standard DGFF estimates (see Subsection~\ref{s:2.2}) then show that the mean at $x$ is $u' + O(1)$ and yields the lower bound.

For the upper bound, we use a different representation of the conditional mean in~\eqref{e:101.45}. Instead of conditioning on the values of $h$ at all vertices in generation $l_u$, we condition on the values of $h$ on all vertices on the branch $([x]_k : k=0, \dots, l_u)$. Letting $\ol{\rmP}_{l_u}$ denote the law of $h$ restricted to such branch, we thus have
\begin{equation}
\label{e:101.48}
\rmE_n \big(h(x)\,\big|\, \Omega_n(u)\big) = \frac{\ol{\rmE}_{l_u} \big(h(x) \varphi_{n,l_u,u}(h)\big)}
	{\ol{\rmE}_{l_u} \varphi_{n,l_u,u}(h)} \,,
\end{equation}
where this time, 
\begin{equation}
\label{e:101.49}
	\varphi_{n,l_u,u}(h) = \prod_{k=0}^{l_u} p^{\frac12 (1+\1_l(k))}_{n-k}\big(u_k-h(k)\big) 
\quad ; \qquad u_k = u-m_n+m_{n-k} \,.
\end{equation}
Above, we have identified between $([x]_k :\: k=0, \dots, l_u)$ and the linear graph $(0,\dots, l_u)$ via the obvious isomporphism. The $k$-th term in the product accounts for the (conditional) probability that the minimal height among all leaves of the ``half'' subtree of $[x]_k$, which exclude the subtree of $[x]_{k+1}$, is at least $-m_n + u$.

As in the lower bound, the first step is to replace the function 	$\varphi_{n,l_u,u}$ in~\eqref{e:101.48} by 
\begin{equation}
	\label{e:104.4}
	\psi_{n,l_u,u}(h)  :=  \prod_{k=1}^{l_u} \tilde{p}^{1+\1_l(k)}_{n-k}(u_k - h(x))
\quad ; \qquad	
\tilde{p}_k(v) := \exp \Big(-\frac{(v^+)^2 + C v}{4-2^{-k+2}} \Big) \,,
\end{equation}
which, thanks to a-priori bounds on the right tail of the minimum and
FKG, makes the desired quantity larger, once $C$ is chosen large enough.
Using the Cameron-Martin Formula to (carefully) tilt the measure $\ol{\rmP}_{l_u}$ in both the numerator and denominator of~\eqref{e:101.48}, the ratio becomes upper bounded by
\begin{equation}
\label{e:101.51}
u' + \frac{\ol{\rmE}_{l_u} \big(h(l_u) \chi_{n,l_u,u}(h))}{\ol{\rmE}_{l_u}\, \chi_{n,l_u,u}(h)} \,,
\end{equation}
where 
\begin{equation}
\label{e:101.52}
\chi_{n,l_u,u}(h) = \exp \Big( -\sum_{k=1}^{l_u} g_{n,k,u} \big((h(k)\big) \Big) \,,
\end{equation}
with the functions $g_{n,k,u}: \bbR \to \bbR$ satisfying 
for some $a, b, D \in (0,\infty)$:
\begin{description}
	\item[A1] $g_{n,k,u}(v) \geq a |v|$ if $|v| > b$, and $|g_k(v)| \leq D$ if $|v| \leq b$,
	\item[A2] $g_{n,k,u}$ is increasing for all $v > b$ and decreasing for $v < -b$. 
\end{description}

As, evidently, $(h(k) :\: k = 0, \dots, l_u)$ under 
$\ol{\rmP}_{l_u}$ is a random walk with standard Gaussian steps, the quotient in~\eqref{e:101.51} can be interpreted as the mean at time $l_u$ of such random walk subject to a {\em localizing potential} $\chi_{n,l_u,u}$, which weakly attracts it to $0$. We then show in Proposition~\ref{p:2.14} that under {\rm (A1)-(A2)}, the height of such walk has uniform exponential tails at all times $k \in [0,l_u]$ and in particular uniformly bounded means. Combined with~\eqref{e:101.52}, this gives the desired upper bound on the conditional mean.

\subsubsection{Reduction to a random walk subject to a localization force}
Deriving sharp asymptotics for the right tail of the centered minimum is also the ``bootstrap step'' in obtaining bounds on the means, fluctuations and covariances of the conditional field (Theorem~\ref{t:1.3} and Corollary~\ref{c:1.4}), as well as asymptotics for its law (see~\cite{Work2}). Indeed, as in the upper bound on the conditional mean in~\eqref{e:101.45}, 
a key step in proving the desired bounds and asymptotics, is the representation of $h$ on $([x]_k :\: k=0, \dots, l_u)$ as a random walk that is subject to a localization force which attracts it to zero. 

Once again, the argument begins by writing
\begin{equation}
\label{e:101.53}
\rmP_n \Big(\big(h([x]_k)\big)_{k=0}^{l_u} \in \cdot \,\big|\, \Omega_n(u)\Big) = \frac{\ol{\rmE}_{l_u} \big(\varphi_{n,l_u,u}(h);\; h \in \cdot \big)}
	{\ol{\rmE}_{l_u} \varphi_{n,l_u,u}(h)} \,,
\end{equation}
with $\varphi_{n,l_u,u}$ as in~\eqref{e:101.49}. However, unlike in the derivation before, we do not replace $\varphi_{n,l_u,u}$ by a dominating function $\psi$, as we are after sharp estimates, not bounds. Instead we crucially use the sharp asymptotics for the right tail function $p_k$ in the product in~\eqref{e:101.49} (with, importantly, the error in the derivative form). 

Proceeding as before by carefully tilting both the numerator and denominator in~\eqref{e:101.53}, we obtain (Proposition~\ref{p:4.1}),
\begin{equation}
\rmP_n \Big( \big(h([x]_k) - u'(1-2^{-k})\big)_{k=0}^{l_u} \in \cdot 
\,\Big|\, 
\Omega_n(u) \Big) = \frac{\ol{\rmE}_{l_u} \big(\chi_{n,l_u,u}(h);\; h \in \cdot \big)}
	{\ol{\rmE}_{l_u}\, \chi_{n,l_u,u}(h)} \,,
\end{equation}
where $\chi_{n,l_u,u}$ is as in~\eqref{e:101.52} for different functions
$g_{n,k,u}$ which still satisfy {\rm (A1)} and {\rm (A2)}. This (as before, but now with equality) recasts the law of $\big(h\big([x]_{k}\big) - (1-2^{-k})u'\big)_{k=0}^{l_u}$ under $\rmP_n(-|\Omega_n(u))$ as a {\em random walk subject to a localizing potential}.

As a consequence of the above representation, good control on the functions  $g_{n,k,u}$ and the Markovian structure of $h$, we can also show (see Proposition~\ref{p:4.2a}) that $\big(h\big([x]_{k}\big) - (1-2^{-k})u'\big)_{k=0}^{l_u}$ under $\rmP_n(-|\Omega_n(u))$  is in fact Markovian with step kernel
\begin{equation}
\label{e:101.56}
	\rmP \big(Y_{k+1} - v \in \cdot
		\, \big|\, Y_k = v \big)
	\quad , \qquad Y_k := h([x]_{k}) - (1-2^{-k})u' \,,
\end{equation}
having (at least) uniformly Gaussian upper and lower tails and drift $d_v$ which satisfies:
\begin{description}
	\item[B1] $\ul{d}_v \geq d$ for all $v < -b$,
	\item[B2] $\ol{d}_v \leq -d$ for all $v > b$,
	\item[B3] $|\ul{d}_v| \vee |\ol{d}_v| < |v| + D$ for all $v \in \bbR$,
\end{description}
for some $b,d,D \in (0,\infty)$,
Thus, an alternative representation of $\big(h\big([x]_{k}\big) - (1-2^{-k})u'\big)_{k=0}^{l_u}$ under $\rmP_n(-|\Omega_n(u))$ is as, what we call with a slight abuse of terminology, a {\em random walk subject to a localizing drift}.

With the above two representations at hand (specialized to $u=m_n$), bounding the mean and fluctuations under the conditioning on $\Omega_n^+$, reduces to controlling the same quantities for a random walk subject to a localization force. This is not a difficult task, as this is model is quite tractable and the theory (of one dimensional random surfaces or directed polymers) is fairly developed. The key tool here is Proposition~\ref{p:2.14}, mentioned before, which shows that for a random walk subject to a localizing potential, the position of the walk at all times $k \leq l_u$ has (at least) uniformly bounded exponential tails. 

To control the covariances and, in the sequel work, to obtain convergence in law, the key tool is Proposition~\ref{p:4.2} which shows that for a random walk $(Y_k)_{k \geq 0}$ subject to a localizing drift which satisfies {\rm (B1)-(B3)}, we have
\begin{equation}
	\big\| \rmP (Y_{k'} \in \cdot\,|\,Y_k = v) - \rmP(Y_{k'} \in \cdot\,|\,Y_k = v')) \big\|_{\rm TV} \leq C\rme^{-(c(k'-k)-|v|\vee |v'|)^+} \,.
\end{equation}
Using the second representation above (\ref{e:101.56} with $u=m_n$), this implies that the conditional field restricted to a branch of the tree ``forgets'' its height exponentially fast in the distance from the root. This is used here to derive the exponential decays of the covariances (Theorem~\ref{e:1.14}) and, in the next work, to show that the centered height at $x \in \bbL_{l_n}$ admits a weak limit.

\subsubsection*{Paper outline}
The remainder of the paper is organized as follows. In Section~\ref{s:2} we state general preliminary results from the theory of Gibbs measures and discrete harmonic analysis. Section~\ref{s:2.5} contains preliminary BRW results. Some of these results come directly from existing literature, while others are developed here.
Section~\ref{s:3} includes the results we need on random walks subject to a localizing potential or a localzing drift. The proof of Theorem~\ref{t:1.1} is the focus of Section~\ref{s:5}. The reduction to a random walk subject to a localizing force is done in Section~\ref{s:6}. Finally, Section~\ref{s:7} provides the proofs of Theorem~\ref{t:1.3}, Corollary~\ref{c:1.4} and Theorem~\ref{t:1.4}.

\section{General preliminaries}
\label{s:2}

\subsection{DGFF and discrete harmonic analysis on general graphs}
\label{s:2.2}
Let $\bbG = (\bbV,\bbE,\omega)$ be a finite non-empty undirected graph with positive edge weights $\omega = (\omega_e)_{e \in \rmE}$.
Let also $\emptyset \subsetneq \bbU \subseteq \bbV$ and $\psi: \bbU \to \bbR$. Then the DGFF on $\bbG$ with boundary values $\psi$ on $\bbU$ is random field $h: \bbV \to \bbR$ with law given by
\begin{equation}
\label{e:2.4}
	\rmP(\rmd h) = \frac{1}{Z} \exp \Big(-\frac12\sum_{\{x,y\} \in \bbE} \omega_{\{x,y\}} \big(h(x) - h(y)\big)^2\Big)
	\prod_{x \in \bbV \setminus \bbU} \rmd h(x) \prod_{x \in \bbU} \delta_{\psi(x)} \big(h(x)\big) \,,
\end{equation}
where $Z$ is the constant which makes the above a probability measure.

There is a strong connection between the DGFF and discrete harmonic analysis on the same underlying graph. Indeed, let 
$\Delta: \bbR^\bbV \to \bbR^\bbV$ be the (discrete, negative) Laplacian on $\bbG$, defined via
\begin{equation}
	-\Delta f(x) = \sum_{y: \: \{x,y\} \in \bbE} \omega_{\{x,y\}} \big(f(y) - f(x)\big) \,.
	\end{equation}
We shall call a function $f: \bbV \to \bbR$ (discrete) harmonic at $x$ if $\Delta f(x) = 0$. Given a function $f$ on a non-empty subset $\bbU \subseteq \bbV$, we denote by $\ol{f}$ the unique real-valued function on $\bbV$ which identifies with $f$ on $\bbU$ and is harmonic at all $x \in \bbV \setminus \bbU$. Letting $S = (S_t)_{t \geq 0}$ be a continuous random walk on $\bbG$ with edge-transition rates $\omega$ (that is $S$ is a Markov Jump Process on $\bbV$ with transition rates $\omega_{\{x,y\}}$ between state $x$ to state $y$) and writing $\rmP_x$ and $\rmE_x$ for the probability measure and expectation under which $S_0 = x \in \bbV$, it is well known that
\begin{equation}
\label{e:2.3a}
\ol{f}(x) = \rmE_x f(S_{\tau_\bbU})
\ , \quad 
\tau_\bbU := \min \big\{t \geq 0 :\: S_t \in \bbU\big\} \,.
\end{equation}

The (discrete, negative) Green Function on $\bbG$ with boundary vertices $\bbU$ is the function
$\cG : \bbV \times \bbV \to \bbR$ which vanishes if either $x$ or $y$ is in $\bbU$ and otherwise satisfies
\begin{equation}
\Delta_y \cG(x,y) = \1_x(y) 
\quad, \qquad x,y \in \bbV \setminus \bbU \,.
\end{equation}
With the definition of $S$ and $\tau_\bbU$ as above, we have
\begin{equation}
\cG(x,y) = \rmE_x \int_{t=0}^{\tau_\bbU} \1_y(S_t) \rmd t \,.
\end{equation}

The connection to the law of the DGFF stems from the fact that the exponential term in~\eqref{e:2.4} can be written as
\begin{equation}
\label{e:2.4a}
	\rme^{-\frac12 \langle h,\, \Delta h \rangle} \,,
\end{equation}
where, henceforth, $\langle \cdot, \cdot \rangle$ is the standard inner product on $\bbR^{\bbV}$.
It then follows by a standard computation that $h$ with law as in~\eqref{e:2.4} is Gaussian with means and covariances given by
\begin{equation}
\label{e:2.7}
\rmE h(x) = \ol{\psi}(x)
\quad, \qquad
\rmCov\big(h(x), h(y)\big) = \cG(x,y) \,.
\end{equation}

The BRW up to generation $n$ corresponds to the DGFF on the graph $\bbG = \bbT_n$, with boundary vertices $\bbU = \{0\}$ and conductances $\omega  \equiv 1$. In the sequel, we shall also consider the case of the BRW conditioned on given uniform values at generation $n$. This case corresponds to the DGFF on the same graph and with the same conductances, but with boundary vertices $\bbU = \{0\} \cup \bbL_n$. If $\psi : \{0\} \cup \bbL_n \to \bbR$ are prescribed boundary values, with $\psi(0) = 0$ and $\psi(z) = u$ for $z \in \bbL_n$ and some $u \in \bbR$, then in view of~\eqref{e:2.3a},
\begin{equation}
\label{e:2.8b}
	\ol{\psi}(x)=u \rho_n(|x|) \,,
\end{equation}
where $\rho_n(k)$ is the probability that a discrete time random walk with steps $\frac13 \delta_{-1} + \frac23 \delta_{+1}$ starting from $k$ reaches $n$ before reaching $0$. The standard Gambler's Ruin Formula gives 
\begin{equation}
\label{e:2.9b}
\rho_n(|x|) = \frac{1-2^{-k}}{1-2^{-n}} = 1 - \frac{2^{n-k} - 1}{2^n-1} \,.
\end{equation}

The following is a simple version of the Cameron-Martin Theorem which follows directly from writing the exponential in the definition of the DGFF as in~\eqref{e:2.4a}.
\begin{lemma}
\label{l:1}
Let $\rmP$ be the law of a DGFF on $\bbG=(\bbV, \bbE, \omega)$ with zero boundary conditions on $\emptyset \subsetneq \bbU \subseteq \bbV$. Let also $\mu \in \bbR^\bbU$. Then for any $F: \bbR^\bbV \to \bbR$,
	\begin{equation}
	\rmE F(h) = \rme^{- \tfrac12 \langle \mu, \Delta \mu \rangle} \rmE \Big( F(h + \mu) \rme^{- \langle h, \Delta \mu \rangle} \Big)\,.
\end{equation}
\end{lemma}

\subsection{Change of measure}
Let $\bbV$ be a finite index set and suppose that $\rmP$ is a probability measure  on $\bbR^\bbV$ representing the law of some field $h$ on $\bbV$. In the sequel we consider two types of measure changes of $\rmP$. Given  $\mu \in \bbR^\bbV$, the first measure change is obtained by considering the law
\begin{equation}
	\rmP (\cdot - \mu) \,,
\end{equation}
namely, the law of $h$ after shifting its mean by $\mu$. We shall refer to this change-of-measure as a {\em tilting} of the law $\rmP$. 

For $\varphi : \bbR^\bbV \to (0, \infty)$, we also define $\rmP^\varphi$ via
\begin{equation}
\label{e:2.10a}
	\frac{\rmd \rmP^\varphi}{\rmd \rmP}  = \Big(\textstyle \int \varphi \rmd \rmP \Big)^{-1} \varphi \,.
\end{equation} 
While both measure changes use the same notation, there will be no ambiguity, as the one to use will be clear from the mathematical type of the superscript. 

\subsection{FKG Property of measures}
\label{s:2.4}
Here and after we endow the space $\bbR^\bbV$ with the usual partial order under which $h_1 \geq h_2$ if and only if $h_1(x) \geq h_2(x)$ for all $x \in \bbV$. 
We shall call a function $F: \bbR^\bbV \to \bbR$ non-decreasing if $\varphi(h_1) \geq \varphi(h_2)$ if $h_1 \geq h_2$. If $\rmP$ and $\rmQ$ are two measures on $\bbR^{\bbV}$ such that $\int \varphi(h) \rmP(\rmd h) \geq \int \varphi(h) \rmQ(\rmd h)$ for all non-decreasing bounded continuous $\varphi$, then we shall say that $\rmP$ stochastically dominates $\rmQ$ and write $\rmP \geq_s \rmQ$. A probability measure $\rmP$ on $\bbR^\bbV$ is said to have the FKG property (is FKG, for short), if for any two bounded measurable functions $\varphi,\psi: \bbR^\bbV \to \bbR$ which are non-decreasing, it holds that $\int \varphi(h)\psi(h) \rmP(\rmd h)  \geq \int \varphi(h) \rmP(\rmd h) \int \psi(h) \rmP(\rmd h)$. Equivalently, $\rmP^\varphi \geq_s \rmP$ for all bounded, positive and continuous $\varphi:\: \bbR^\bbV \to \bbR$ which is non-decreasing.

\begin{remark}
\label{r:2.1}
The restriction to bounded functions in both stochastic domination and the FKG property can be removed if one considers non-negative test functions. This follows straightforwardly from the Monotone Convergence Theorem. Similarly, the continuity requirement can be removed, since bounded continuous non-decreasing functions are dense in all bounded non-decreasing functions under the $\bbL^1$ topology, as shown in~\cite[Lemma]{Pitt82}.
\end{remark}

\begin{remark}
\label{r:2.2}
It is well known that any probability measure on $\bbR$ is FKG. In particular, if $\rmP$ and $\rmQ$ are two such measures which are absolutely continuous with respect to Lebesgue measures and have strictly positive Radon-Nikodym derivatives, then $\log \frac{\rmP(\rm du)}{\rmd u} \geq \log \frac{\rmQ(\rmd u)}{\rmd u}$ implies $\rmP \geq_s \rmQ$.	
\end{remark}

Positivity of the covariances of $h$ under $\rmP_n$ shows that it is FKG~\cite[Corollary~1.7]{Herbst1991}. The following lemma shows that this is also the case under conditioning on $\Omega_n(u)$ for any $u \in \bbR$. The proof is essentially contained in~\cite[Lemma~3.1]{Deuschel1996}.
\begin{lemma}\label{lemma:FKG}
For all $u \in \bbR$, the law 
	$\rmP_n\big(-\,\big|\, \Omega_n(u)\big)$ is FKG. In particular, $\rmP_n^+$ is FKG.
\end{lemma}
\begin{proof}
We use an approximation argument. Let $u \in \bbR$ and for $\beta>0$, define
\begin{equation}
\varphi_{u,\beta}(h) :=  \exp \Big(-\beta \sum_{x\in \bbL_n} \big((h(x)-u)^-\big)^3 \Big).
\end{equation}
Since $\varphi_{u,\beta}$ is twice continuously differentiable, it follows from~\cite[Corollary~1.7]{Herbst1991} that $\rmP_n^{\varphi_{u,\beta}}$ is FKG.
It remains to take $\beta \to \infty$ and observe that $\varphi_{u, \beta}(h) \to 1_{\Omega_n(u)}(h)$ and that $\varphi_{u,\beta}$ is always bounded by $1$.
\end{proof}

As a consequence of FKG we have the following
\begin{lemma}
\label{l:2.4}
Let $\varphi: \bbR^{\bbT_n} \to \bbR_+$ be continuous, non-decreasing and positive. Then for all $u \in \bbR$, the function
\begin{equation}
v \mapsto \rmE_n \big(\varphi(h+v)\,\big|\, h \in \Omega_n(u-v) \big) \,,
\end{equation}
is non-decreasing in $v$.
\end{lemma}
\begin{proof}
Fix $u \in \bbR$ and for $v \in \bbR$, let $\rmP^v_n$ be the law of $h+v$ under $\rmP_n$. Then, the desired statement is equivalent to 
$v \mapsto \rmP^v_n(-|\Omega_n(u))$ being monotonically increasing in $v$. Since $\rmP^v_n(-|\Omega_n(u))$ is FKG by Lemma~\ref{lemma:FKG}, and in view of Remark~\ref{r:2.1}, it is sufficient to show that for $v' > v$, the Radon-Nikodym derivative 
\begin{equation}
\label{e:2.13}
	\frac{\rmP^{v'}_n\big(\rmd h \,|\, \Omega_n(u) \big)}{\rmP^{v}_n\big(\rmd h\, |\, \Omega_n(u)\big)}
\end{equation}	
is a non-decreasing continuous positive function of $h$ on $\Omega_n(u)$. Indeed, the last display is equal to
\begin{multline}
	\frac{\rmP_n(\Omega_n(u-v))}{\rmP_n(\Omega_n(u-v'))}\exp \Big\{-\frac12 \sum_{x \in \bbL_1} \Big((h(x)-v')^2 - (h(x)-v)^2\Big)\Big\}
	 = C \exp \Big\{ (v' - v) \sum_{x \in \bbL_1} h(x) \Big\} \,,
\end{multline}
for some $C > 0$ which depends on $u$,$v$, $v'$ and $n$. The right hand side is non-decreasing, continuous and positive as desired.
\end{proof}

\section{Branching random walk preliminaries}
\label{s:2.5}
In this section we collect standard results concerning the BRW. While some of them are readily available from the literature, a few require a short derivation, which we include here.
\subsection{Tail estimates for the minimum}
\label{s:2.5.2}
Theorem~\ref{t:1.1}, which is one of the main results in this work, provides sharp asymptotics for the upper tail of the law of the centered minimal position of $h$ on $\bbL_n$ for values which are allowed to depend on $n$. While of interest by itself (e.g., it gives sharp asymptotics for $\rmP_n(\Omega_n^+))$ as a corollary), it is also a key input in the proofs of other results in this paper. The proof of this theorem as well as other results in this work, rely on existing coarser/less general estimates for the tails of the centered minimum. This subsection is therefore devoted to presenting these estimates and using them to derive a priori more refined versions thereof.

Henceforth, for $n \geq 0$, we define the right tail function of the minimum $p_n : \bbR \to [0,1]$ via
\begin{equation}
	\label{e:3.3}
	p_{n}(u) := \rmP_n\big(\Omega_n(u)\big) 
	\ , \ \ u \in \bbR \,.
\end{equation}
The following upper tail estimate is essentially taken from~\cite{Roy_2024}. The proof goes by comparison of the centered minimum of the BRW with the centered minimum of a variant thereof, known as the {\em Switching Sign Branching Random Walk} (SSBRW). This is $\zeta_n$ in the statement below.
\begin{lemma}\cite[Lemma~3.2 and Lemma~4.1]{Roy_2024}
\label{l:2.2}
For all $n \geq 1$,
\begin{equation}
\label{e:2.20}
\min_{x \in \bbL_n} h(x) + m_n  \overset{\rmd}= \xi_n + \zeta_n \,,
\end{equation}
where $\xi_n$ is a centered Gaussian with variance $1-2^{-n}$ and $\zeta_n$ is some random variable which is independent of $\xi_n$. Moreover, 
there exist constants $C,c \in (0,\infty)$ and $\delta: \bbR \to (0,1)$ satisfying $\delta(u) \to 1$ as $u \to -\infty$, such that for all $n \geq 1$, $u \in \bbR$,
\begin{equation}
\label{e:2.21}
\delta(u) \leq \rmP (\zeta_n \geq u) \leq C \rme^{-\rme^{cu}} \,.
\end{equation}
\end{lemma}
\begin{proof}
This is essentially Lemma~3.2 and Lemma~4.1 in~\cite{Roy_2024}. The proof for the upper bound when $u \leq m_n$ can be taken verbatim from the proof of Lemma~4.1 in~\cite{Roy_2024} (the statement therein includes the restriction that $u=o(n)$, but the proof goes through for all $u \leq m_n$. For $u > m_n$, the probability in~\eqref{e:2.21} is zero, as, by construction, there always exists a leaf whose value under the SSBRW is non-positive. The lower bound follows immediately from the tightness of the left hand side in~\eqref{e:2.20}.
\end{proof}

As a consequence, we have the following bounds on the derivative of $p_n$.
\begin{lemma}
\label{l:4.2a}
There exist $C \in (0,\infty)$ such that for all $n \geq 1$ and $u \in \bbR$,
\begin{equation}
	\label{e:4.10}
	-\frac{1}{1-2^{-n}} u^+ - C \leq \frac{\rmd}{\rmd u} \log  p_n(u) \leq -\frac{1}{1-2^{-n}} u^+ + C \log (u\vee \rme) .
\end{equation}
In particular,
\begin{equation}
c \exp \Big(-\tfrac{1}{2-2^{-n+1}} (u^+)^2\Big)
\leq p_n(u) \leq  C \exp \Big(-\tfrac{1}{2-2^{-n+1}} (u^+)^2 + C u \log (u\vee \rme)\Big) \,.
\end{equation}
\end{lemma}
\begin{proof}
Using Lemma~\ref{l:2.2}, may write
\begin{equation}
	p_n(u) = \rmP \big( \xi_n + \zeta_{n} \geq u \big)
	= \int  \rmP(\zeta_{n} \geq u-s)  \rmP (\xi_n \in \rmd s)  = \int \rmP(\zeta_n \geq -s)  \rmP (\xi_n - u \in \rmd s) \,,
\end{equation}
where $\xi_n$ and $\zeta_n$ are as in the lemma. Taking the derivative and recalling that $\xi_n$ is a centered Gaussian with variance $1-2^{-n}$, 
\begin{align}
	\frac{\rmd}{\rmd u} p_n(u) & = - \int \frac{u+s}{1-2^{-n}} \rmP(\zeta_n \geq -s) \rmP (\xi_n - u \in \rmd s)  \\
	& = 
	- \frac{1}{1-2^{-n}} \int  s \rmP(\zeta_n \geq u-s) \rmP (\xi_n  \in \rmd s) 
	= - \frac{1}{1-2^{-n}} \rmE \big( \xi_n ; \xi_n + \zeta_n \geq u \big) \,,
\end{align} 
so that the derivative in~\eqref{e:4.10} becomes 
\begin{equation}
	\label{e:4.14}
\frac{1}{p_n(u)}\frac{\rmd}{\rmd u} p_n(u) =	 - \frac{1}{1-2^{-n}} \frac{\rmE \big( \xi_n ; \xi_n + \zeta_n \geq u \big) }{\rmP(\xi_n + \zeta_n \geq u )} = -\frac{1}{1-2^{-n}} \rmE \big( \xi_n \big| \xi_n + \zeta_n \geq u \big) \,.
\end{equation}

Assume first that $u \geq 1$. To lower bound the conditional mean, we may add the restriction that $\zeta_n \leq C \log u$ for some $C$ to be chosen later. On this event, $\xi_n$ must be at least $u-C \log u$, so that the conditional restricted mean is at least
\begin{multline}
	\label{e:4.15}
	(u-C\log u) \rmP (\zeta_n \leq C \log u | \xi_n+\zeta_n \geq u) \\
	\geq (u-C\log u) - u \rmP (\zeta_n > C \log u | \xi_n+\zeta_n \geq u) .
\end{multline}
But since
\begin{equation}
	\label{e:4.16}
	\rmP ( \xi_n+\zeta_n \geq u) \geq \rmP ( \xi_n \geq u) \rmP(\zeta_n \geq 0) \geq c \rme^{-\frac{u^2}{2-2^{-n+1}}} \,,
\end{equation}
the double exponential tail of $\zeta_n$ shows that the second term on the right hand side of~\eqref{e:4.15} is $o(1)$ in $s$ once $C$ is chosen large enough.
At the same time, the conditional mean in~\eqref{e:4.14} is upper bounded by
\begin{equation}
	u+1 + \frac{\rmE \big(\xi_n ; \xi_n \geq u +1)}{\rmP(\xi_n + \zeta_n \geq u \big)}.
\end{equation}
Now, a standard computation shows that the numerator above is at most a constant times $\rme^{-(u+1)^2/(2-2^{-n+1})}$. Combined with the lower bound in~\eqref{e:4.16}, this shows that the last fraction is $o(1)$ in $u$ and thus the conditional mean in~\eqref{e:4.14} is at most $u+C$.

Now, if $u < 1$, we lower bound the conditional mean in~\eqref{e:4.14} by
\begin{equation}
	\frac{\rmE \xi_n - \rmE \big(|\xi_n| ; \xi_n + \zeta_n < u\big)}{\rmP(\xi_n + \zeta_n \geq u)} \geq 
	- \frac{\rmE |\xi_n|}{\rmP(\xi_n + \zeta_n \geq u)} \,,
\end{equation}
which is bounded below by a constant thanks to~\eqref{e:4.16}. Similarly, we can upper bound the conditional mean in~\eqref{e:4.14} by 
\begin{equation}
	\frac{\rmE |\xi_n|}{\rmP(\xi_n + \zeta_n \geq u)} 
\end{equation}
and proceed as before. Collecting all bounds we recover the first statement. The second follows by integration and the fact the $\log p_n(0)$ is uniformly bounded.
\end{proof}

Next, we recall that thanks to the convergence of the centered minimum~\eqref{e:01.13},  the limit
\begin{equation}
\label{e:2.32}
	p_\infty(u) = \lim_{n \to \infty} p_{n}(u) \,,
\end{equation}
exists for all $u \in \bbR$ and is non-trivial. As the limit is continuous in $u$ (being the sum of two independent random variables, one of which at least is absolutely continuous w.r.t. the Lebesgue measure) and $p_n$ is monotone for all $n$, the convergence is even uniform on compact sets. The next lemma shows that the same applies to the derivative.
\begin{lemma}\label{l:2.9a}
	It holds that 
	\begin{equation}
		\label{e:2.33}
		\lim\limits_{n\to \infty}\frac{\mathrm{d}}{\mathrm{d}u} \log p_n(u)= \frac{\mathrm{d}}{\mathrm{d}u} \log p_\infty(u)
	\end{equation}
	uniformly in $u$ on compact sets. For all $n \in [0, \infty]$, the derivatives 
	$\frac{\mathrm{d}}{\mathrm{d}u} \log p_n$ are continuous and satisfy
	\begin{equation}
		\label{e:2.37b}
		-2u^+ - C \leq \frac{\mathrm{d}}{\mathrm{d}u} \log p_n(u) \leq 0 \,.
	\end{equation}
\end{lemma}
\begin{proof}
Writing $p_n(u)$ as $\rmP_n(-u+h \in \Omega_n(0))$. It follows by differentiating the exponent in~\eqref{e:2.4} that 
	\begin{equation}
		p'_n(u) = -2\rmE_n \big(h(x) \;;\; -u+h \in \Omega_n(0) \big) \,,
	\end{equation}
	where $x \in \bbL_1$, so that the derivative on the left hand side of~\eqref{e:2.33} is equal to
	\begin{equation}
		\label{e:2.35a}
		\frac{p'_n(u)}{p_n(u)} = -2\rmE_n \big(h(x) \,|\, \Omega_n(u)\big) = -2\rmE_1^{\varphi_{n,u}}(h(x))\,,
	\end{equation}
	where $\varphi_{n,u}(h) := p_{n-1}(u - m_n + m_{n-1} - h(x))$.
	
	Now it follows from $m_n - m_{n-1} \to c_0$ as $n \to \infty$, the continuity of $p_\infty$ and uniform convergence in~\eqref{e:2.32}, that 
	$\varphi_{n,u}(h) \to \varphi_{\infty,u}(h) := p_\infty(u-c_0 - h(x))$ as $n \to \infty$ for all $h$. Since $\varphi_n$ is bounded by one and $h(x)$ is a standard Gaussian, it follows from the Dominated Convergence Theorem that the right hand side in~\eqref{e:2.35a} tends to $-2\rmE_1^{\varphi_{\infty,u}} h(x)$. 
	As $\varphi_{n, u}$ is continuous in $u$ and bounded by $1$, it follows again by the Dominated Convergence Theorem that $-2\rmE_1^{\varphi_{n,u}} h(x)$ is continuous in $u$ for all $n \in [0,\infty]$. 
	
	Thanks to Lemma~\ref{lemma:FKG} we know that the law $\rmP_n(-\,|\, \Omega_n(u))$ is FKG for all $u \in \bbR$. This implies in particular that $\rmE_n \big(h(x) \,|\, \Omega_n(u)\big)$ is monotone increasing in $u$ and therefore that the right hand side of~\eqref{e:2.35a} is decreasing in $u$. It follows that the convergence is uniform on compacts as desired. This implies in particular that the left-hand side in~\eqref{e:2.33} converges uniformly on compacts. This allows to interchange derivative and limit, which implies \eqref{e:2.33}. ~\eqref{e:2.37b} follows from the first part of~\eqref{e:2.35a} via 
	Proposition~\ref{l:2.8a} for all $n < \infty$ and therefore also for $n = \infty$.
\end{proof}

We shall also need estimates and asymptotics for the left tail of the centered minimum. To this end, we define for $n \geq 1$ and $u \in \bbR$, 
\begin{equation}
\label{e:103.24}
	q_n(u) := 1 - p_n(-u) = \rmP_n \Big(\min_{\bbL_n} h \leq -m_n - u \Big) \,.
\end{equation}
We then have,
\begin{lemma}[\cite{bramsondingoferbrw}, Proposition~3.1]
\label{l:2.10}
There exist $C,c \in (0,\infty)$ such that for all $n \geq 1$, $u \geq 0$,
\begin{equation}
q_n(u)
\leq C(u+1) \rme^{-c_0 u}  \,, 
\end{equation}	
and for all $n \geq 1$, $u \in [0, n^{1/2}]$,
\begin{equation}
q_n(u)
\geq c (u+1) \rme^{-c_0 u} \,.
\end{equation}	
Moreover, there exists $C_0 \in (0,\infty)$ such that
\begin{equation}
\lim_{u \to \infty} \limsup_{n \to \infty}
\Bigg|\frac{q_n(u)}{C_0 u \rme^{-c_0 u}} - 1 \Bigg| = 0 \,.
\end{equation}	
\end{lemma}

\subsection{Two representations of the conditional measure}
In the sequel, we will make frequent use of the following two representations of the measures $\rmP_n(-|\Omega_n(u)$ for $u \in \bbR$.
For what follows,  the subtree of $\bbT$ rooted at $x \in \bbT$ will be denoted by $\bbT(x)$, with the restriction of which to vertices at distance at most and exactly $n$ from $x$ denoted by $\bbT_n(x)$ and $\bbL_n(x)$ respectively.

\subsubsection{Doob's $h$-transformed BRW}
The first representation follows by conditioning on the values of $h$ on $\bbL_k$ for $k \in [0,n]$.  
Indeed, then the Gibbs-Markovian structure of the field gives for each $u \in \bbR$,
\begin{equation}
\label{e:2.41e}
	p_n(u) = \rmE_k \varphi_{n,k,u}(h)
\end{equation}
where
\begin{equation}
\label{e:2.40d}
\varphi_{n,k,u}(h) := \prod_{x \in \bbL_k} p_{n-k} \big(-h(x) + u - m_n + m_{n-k} \big) \,.
\end{equation}
In particular, we get the representation
\begin{equation}
\label{e:2.32d}
	\rmP_n \big(h_{{\bbT_k}} \in \cdot \,\big|\, \Omega_n(u)\big) = \rmP_k^{\varphi_{n,k,u}}(\cdot) \,,
\end{equation}
where we recall the measure change notation from~\eqref{e:2.10a}. Notice that $m_n - m_{n-k} = c_0 k + O(k/n)$.

\subsubsection{Random walk under a pinning potential}
\label{s:2.6.2}
Turning to the second representation, for $n \geq 0$, let us denote by 
$\wh{\bbT}$ the ``half'' subtree of $\bbT$ obtained by removing one (arbitrary) child of the root and all of its descendants. $\wh{\bbT}_n$ and $\wh{\bbL}_n$ will be defined analogously. 
For $u \in \bbR$, we define the event $\wh{\Omega}_n(u) \subset \bbR^{\wh{\bbT}_n }$ as
\begin{equation}
	\wh{\Omega}_n(u) := \Big \{ \min_{x \in \wh{\bbL}_n} h(x) \geq -m_n + u \Big\}
\end{equation}
and set
\begin{equation}
\label{e:102.44}
	\hat{p}_n(u) = \rmP_n\big(h_{{\wh{\bbT}_n}} \in \wh{\Omega}_n(u) \big)  \,.
\end{equation}
By independence of the restriction of $h$ to the two direct sub-trees of the root, we have
\begin{equation}
\label{e:102.45}
	p_n(u) = \big(\hat{p}_n(u)\big)^2 \,.
\end{equation}

Now, consider the branch $([x]_k)_{k=0}^l$ for some $x \in \bbL_l$. This subgraph of $\bbT$ can then be identified with the linear graph $[0,l] =\{0, \dots, l\}$, via the natural isomorphism $k \leftrightarrow [x]_k$. Consequently the restriction of $h$ to $([x]_k)_{k=0}^l$ has the law of a DGFF on $[0,l]$ with $0$ boundary condition at $0$, which we shall henceforth denote by $\ol{\rmP}_l$. This is also the law of the trajectory of a random walk with centered Gaussian steps, starting from $0$ at time $0$ and run until time $l$. If $n \geq l$ and $u \in \bbR$, we can use this branch to write $\Omega_n(u)$ as 
\begin{equation}
	\bigcap_{k=0}^{l-1} \Big\{h_{{\bbT_{n-k}([x]_k) \setminus \bbT_{n-k-1}([x]_{k+1})}} \in \wh{\Omega}_{n-k}\big(u-m_n+m_{n-k}\big)\Big\} 
	\cap \Big\{h_{{\bbT_{n-l}([x]_l)}}\in \Omega_{n-l}\big(u-m_n+m_{n-l}\big)\Big\} \,.
\end{equation}
It follows from the Gibbs-Markovian structure of $h$ that the measure $\rmP_n(-\,|\,\Omega_n(u))$ can be written, via Definition~\ref{e:2.10a}, as
\begin{equation}
\label{e:2.43a}
	\rmP_n \Big(\big(h([x]_k\big)_{k=0}^l \in \cdot \, \Big|\,\Omega_n(u)\Big) = \ol{\rmP}_l^{\varphi_{n,l,u}}\big(\cdot\big)\,,
\end{equation}
where this time,
\begin{equation}
\label{e:2.38a}
	\varphi_{n,l,u}(h) :=  \prod_{k=1}^{l} \big(\hat{p}_{n-k}\big(u-h(k)-m_n+m_{n-k}\big)\big)^{1+\1_l(k)}
		 = \prod_{k=1}^{l} p^{\frac12 (1+\1_l(k))}_{n-k}\big(u-h(k)-m_n+m_{n-k}\big) \,.
\end{equation}
Observe that the products above may just as well start from $1$ instead of $0$ as the term corresponding to $k=0$ is a constant that gets cancelled with the normalization of the measure. In the sequel we shall sometimes write $\varphi_{n,u}$ if $n=l$ or omit some of the subscripts of $\varphi_{n,l,u}$ to ease some notational burden.

\subsection{Bounds under typical conditioning on the minimum}
In this subsection we include versions of standard estimates for the BRW, albeit where the process is conditioned on having its minimum being above a typical level, namely under the conditioning on $\Omega_n(u)$ for a fixed $u \in \bbR$. The first lemma is the key tool in the proof for such results.
\begin{lemma}
\label{l:103.10}
Let $n \geq 1$, $x_n \in \bbL_n$ and $\xi_n$ measurable w.r.t. $h$ on $\bbT_n$. For each $k \in [0,n]$ set $x_k := [x]_k$ and
\begin{equation}
\xi_k := \rmE_n \big(\xi_n\,\big|\, h\big(\{x_k\} \cup \{\bbT_n \setminus \bbT(x_k)\} \big)\big) \,.
\end{equation}
Then for all $m \geq 0$,
\begin{equation}
\label{e:203.52}
	\Big|\rmE_n \big(\xi_n\,\big|\, \Omega_n(u)\big) - \rmE_n\big( \xi_m |\Omega_n(u\big) \Big| 
	\leq C\big(\rme^{C u^2} + 1\big) \sum_{k=m}^n 
		k \rme^{-c_0 k} \rmE_n \Big((h(x_k)^++1) \rme^{-c_0 h(x_k)} |\xi_k|\Big)
\end{equation}
\end{lemma}
\begin{proof}
For all $u \in \bbR$ and $1 \leq m < n$, set
\begin{equation}
	\Omega_n^m(u) := \Big\{\min \big\{h(z) :\: z\in \bbL_n \setminus \bbL_{n-m-1}(x_{m+1}) \big\} \geq -m_n + u \Big\} \,,
\end{equation}
with $\Omega_n^n(u) \equiv \Omega_n(u)$.  Then,
\begin{equation}
\label{e:202.65b}
\begin{split}
\Big|\rmE_n \big(\xi_n ;\, \Omega_n^m(u) \big) - \rmE_n \big(\xi_n;\, \Omega_n(u) \big)\Big|
& = \Big|\rmE_n \big(\xi_n ;\, \Omega_n^m(u) \setminus \Omega_n^n(u) \big)\Big| \\
& \leq \sum_{j=m+1}^{n}
	\Big|\rmE_n \big(\xi_n ;\, \Omega_n^{j-1}(u) \setminus \Omega_n^{j}(u) \big)\Big| \\
& \leq \sum_{j=m+1}^n \rmE_n 
	\Big(\big|\xi_j\big| \Big(1-\hat{p}_{n-j}(-m_n + m_{n-j} - h(j) + u)\Big)\Big)\Big) \,,
\end{split}
\end{equation}	
where we recall the definition of $\hat{p}_n$ from~\eqref{e:102.44}. 
For the last step we conditioned on $h$ on $\{x_j\} \cup \bbT \setminus \bbT(x_j)$ 
 and used that
\begin{equation}
\Omega_n^{j-1}(u) \setminus \Omega_n^{j}(u) \subseteq 
	\Big\{\min \big\{h(z) :\: z \in \bbL_{n-j}(x_j) \setminus \bbL_{n-j-1}(x_{j+1}) \big\} \leq -m_n + u \Big\} \,,
\end{equation}
with the event on the left hand side independent of $\xi_n$, under this conditioning.

Using that $\hat{p}_n(u) = p^{1/2}_n(u) \geq p_n(u)$ and $-m_n+m_{n-j}\geq c_0 j$ and also Lemma~\ref{l:2.10}, we may upper bound the second term inside the last mean by  a constant times
\begin{equation}
	 \rmE_n \Big( \big(h(x_j)+c_0 j-u\big)^++1\Big) \exp \Big(-c_0 \big(h(x_j) + c_0 j - u)\big)^+\Big)
\leq C' \rmE_n \big(\rme^{c_0 u} + 1\big) j \rme^{-c_0 j} (h(x_j)^++1) \rme^{-c_0 h(x_j)} \,.
\end{equation}
This gives
\begin{equation}
\label{e:202.65a}
\Big|\rmE_n \big(\xi_n ;\, \Omega_n^m(u) \big) - \rmE_n \big(\xi_n;\, \Omega_n(u) \big)\Big|
\leq 
C \big(\rme^{c_0 u} + 1\big) \sum_{j=m+1}^{n-1}  j \rme^{-c_0 j} \rmE_n \big((h(x_j)^++1) \rme^{-c_0 h(j)} |\xi_j|\big) \,.
\end{equation}	
Using this bound with $\xi_m$ in place of $\xi_n$ and invoking the Tower Property, we obtain
\begin{equation}
\label{e:203.62}
\Big|\rmE_n \big(\xi_m ;\, \Omega_n^m(u) \big) - \rmE_n \big(\xi_m;\, \Omega_n(u) \big)\Big|
\leq 
C \big(\rme^{c_0 u} + 1\big) \sum_{j=m+1}^{n-1}  j \rme^{-c_0 j} \rmE_n \big((h(x_j)^+ +1)\rme^{-c_0 h(j)} |\xi_m|\big) \,.
\end{equation}

Using the Tower Property once more, the last mean is at most 
\begin{multline}
	\rmE_n \big((h(x_m)^+ + 1)\rme^{-c_0 h(x_m)} |\xi_m|\big)\, \rmE_n \big((h(x_j) - h(x_m))^+\big) \rme^{-c_0( h(x_j) - h(x_m))} \\\leq (j-m + 1) \rme^{(c_0^2/2)(j-m)} \rmE_n \big((h(x_m)^+ + 1)\rme^{-c_0 h(x_m)} |\xi_m|\big)\,.
\end{multline}
As $c_0 > c_0^2/2$, 
this makes the sum in~\eqref{e:203.62} at most $\rmE_n \big((h(x_m)^+ + 1)\rme^{-c_0 h(x_m)} |\xi_m|\big)$ times 
\begin{equation}
	(m+1) \rme^{-c_0 m} \sum_{j=1}^\infty (j+1)^2 \exp \big((c_0^2/2) - c_0)j\big)
\leq C m \rme^{-c_0 m} \,,
\end{equation}
and upper bounds the sum of the right hand sides of~\eqref{e:202.65a} and~\eqref{e:203.62} by the right hand side of~\eqref{e:203.52}. Using the Tower Property one last time, we may equate the first mean in the left hand sides of~\eqref{e:202.65a} and~\eqref{e:203.62}. Using the Triangle Inequality and dividing by $\rmP_n(\Omega_n(u))$ which is at least $c\rme^{-C(u^+)^2}$ by Lemma~\eqref{l:4.2a} gives the desired bound.
\end{proof}

\begin{proposition}
\label{l:2.8a}
There exists $C \in (0,\infty)$ such that for all $u \in \bbR$, $n \geq 1$, $x \in \bbT_n$,
\begin{equation}
\label{e:2.37a}
0 \leq \rmE_n \big(h(x) \,\big|\, \Omega_n(u)\big) < u^+ + C
\end{equation}
\end{proposition}
\begin{proof}
Thanks to FKG of $\rmP_n^+$ and Lemma~\ref{l:2.4},
\begin{equation}
0 = \rmE_n h(x) \leq \rmE_n \big(h(x) \,\big|\, \Omega_n(u)\big) \leq 
\rmE_n \big(h(x) \,\big|\, \Omega_n(u^+)\big) \leq
u^+ + \rmE_n \big(h(x) \,\big|\, \Omega_n(0)\big) \,.
\end{equation}
Hence we just need to show the upper bound in the case $u=0$.
To this end, for $0 \leq j \leq n$, let $x_n \in \bbL_n$, $x_j := [x_n]_j$ and $\xi_n^{(j)} := h(x_j)$. Notice that for $0 \leq k \leq n$,
\begin{equation}
	\xi^{(j)}_k  := \rmE_n \big(h(x_j)\,\big|\, h\big(\{x_k\} \cup \{\bbT_n \setminus \bbT(x_k)\} \big) = h(x_{j \wedge k}) \,.
\end{equation}
Then the mean in the sum in~\eqref{e:203.52} of Lemma~\ref{l:103.10}, with $\xi^{(j)}_k$ in place of $\xi_k$, is at most
\begin{multline}
	\rmE_n \Big((h(x_{k \wedge j})^+ + 1)(h(x_{k}) - h(x_{k \wedge j})^+ +1) 
		\rme^{-c_0 h(x_{k \wedge j}) - c_0 (h(x_k) - h({k \wedge j}))} |h(x_{k \wedge j})|\Big) \\
\leq C \big((k-j)^++1\big) \rme^{-(c_0^2/2) (k-j)} 
	\rmE_n \Big( |h(x_{k \wedge j})| + 1\big) \rme^{-c_0 h(x_{k \wedge j})}\Big)
\leq C' k^3 \rme^{-(c_0^2/2) k} \,,
\end{multline}
where we used the independent increments of $(h(x_k))_{k=0}^n$ under $\rmP_n$.
Since $c_0^2/2 < c_0$, the sum itself is uniformly bounded for all $n$ and $j$.
Invoking Lemma~\ref{l:103.10} with $m=u=0$ and using that $\rmE_n\big(\xi^{(j)}_0\,\big|\, \Omega_n(u)\big) = 0$, the result follows.
\end{proof}

Next we control the conditional variance.
\begin{proposition}
\label{l:2.13}
There exists $C < \infty$ such that for all $n \geq 1$, $v \in \bbR$ and $x \in \bbT_n$,
\begin{equation}
\big| \rmE_n \big(h(x)^2\,\big| \Omega_n(u) \big) - |x| \big| \leq C(u^2 + 1) \,.
\end{equation}
\end{proposition}
\begin{proof}
For each $1 \leq j \leq n$, let $x_n \in \bbL_n$ and set $x_j := [x_n]_j$ and $\xi_n^{(j)} := h(x_j)^2 - j$. Observe that for all $k \leq n$,
\begin{equation}
\xi_k^{(j)} := \rmE_n \Big(h(x_j)^2 - j\,\big|\, h\big(\{x_k\}\cup \{\bbT_n \setminus \bbT(x_k)\} \big)\Big)
= h(x_{k\wedge j})^2 - (k \wedge j) \,.
\end{equation}
As in the proof of Proposition~\ref{l:2.8a}, the mean on the right hand side of~\eqref{e:203.52}  in Lemma~\ref{l:103.10} is computed to be at most $Ck^4\rme^{-(c_0^2)k/2}$. Since $c_0^2/2 < c_0$, it follows from Lemma~\ref{l:103.10} that there exists $C < \infty$ such that whenever $m > C u^2$ and $j \geq m$,
\begin{equation}
\label{e:103.52}
	\big|\rmE_n\big( h(x_j)^2\,\big|\, \Omega_n(u)\big) - j \big| 
	\leq \rmE_n\big( h(x_m)^2\,\big|\, \Omega_n(u)\big) + m + C' \,.
\end{equation}

At the same time, by FKG and Lemma~\ref{l:2.4},
\begin{equation}
\rmE_n\big( \big(h(x_k)^+\big)^2\,\big|\, \Omega_n(u)\big) \leq  
\rmE_n\big( \big(u^++h(x_k)^+\big)^2\,\big|\, \Omega_n(u)\big)
\leq \frac{2((u^+)^2) + 2\rmE_n \big(h(x_k)^+\big)^2}
{\rmP_n\big(\Omega_n(0)\big)}
\leq C''\big((u^+)^2+k)
\end{equation}
and
\begin{equation}
 \rmE_n\big( \big(h(x_k)^-\big)^2\,\big|\, \Omega_n(u)\big) \leq  
 \rmE_n \big(h(x_k)^-\big)^2 = \frac{k}{2}.
\end{equation}
So that 
\begin{equation}
\label{e:203.55}
	\rmE_n\big( \big(h(x_k)\big)^2\,\big|\, \Omega_n(u)\big) \leq C''((u^+)^2 + k) \,.
\end{equation}

Now, given $n \geq 1$ and $u \in \bbR$, we let $m := \lceil C u^2 \rceil$. If $j \leq m$, we use~\eqref{e:203.55} with $k=j$. Otherwise, we use~\eqref{e:203.55} with $k = m$ and combine with~\eqref{e:103.52}. The result follows.
\end{proof}

\section{Random walk subject to a localization force}
\label{s:3}
In this subsection we collect results for a random walk which is weakly attracted to $0$. We shall consider two types of attractions. One represented as a measure change of the law of a random walk without drift via a pinning potential, and one by imposing a drift in the opposite direction, whenever the process is sufficiently away from zero.

\subsection{Localization via a potential}
\label{ss:3.1}
We recall that $\ol{\rmP}_l$ is the law of a random walk with standard Gaussian steps run until time $l \geq 1$ started from zero. Alternatively, it is a DGFF on the linear graph $[0,l]$ with unit conductances and $0$ boundary conditions at $0$. To specify a different boundary condition $u$ at $0$ we shall do so via formally conditioning on $\{h(0) = u\}$, i.e. we shall write $\ol{\rmP}_l(\cdot\,|\, h(0) = u)$.

As a potential we take a sequence of non-negative functions $g_k : \bbR \to \bbR$ for $k = 1, \dots ,l$. We shall assume that there are $a, b, D \in (0,\infty)$ such that for all $k \in [1,l]$ the following holds:
\begin{description}
	\item[A1] $g_k(u) \geq a |u|$ if $|u| > b$, and $|g_k(u)| \leq D$ if $|u| \leq b$ \,.
	\item[A2] $g_k$ is increasing for all $u > b$ and decreasing for $u < -b$. 
\end{description}
Letting $\psi: \bbR^{[0,l]} \to \bbR$ be defined via
\begin{equation}
	\psi(h) := \exp \Big(-\sum_{k=1}^l g_k\big(h(k)\big)\Big) \,,
\end{equation}
we shall consider the process $h$ under the law  $\ol{\rmP}^\psi_l$, which we shall refer to as a {\em random walk under potential} $\psi$.

The result that we shall need  is included in the next proposition, which shows that the height of the walk at all times is exponentially tight.
\begin{proposition}
\label{p:2.14}
T	here exists $C,c \in (0, \infty)$ which depend only on $a$,$b$ and $D$, such that for all $l \geq 1$, $u \geq 0$,
	\begin{equation}
		\ol{\rmP}^{\psi}_l \big( |h(l)| > u \big) \leq C \rme^{-cu} \,.
	\end{equation}
\end{proposition}

\begin{proof}
By the Union Bound and symmetry, it is sufficient to show exponential decay for the upper 
tail probability of $h(l)$. Also, since adding a constant to $g_k$ does not change the measure, we may assume that all $g_k$-s are positive. 
	Let
	\begin{equation}
		\tilde{\psi}(h) := \exp \Big(- \sum_{k=1}^{l} \Big(g_k\big(h(k)\big) \1_{\{h(k) \leq b \text{ or } k =l \}}  \Big)\Big) \,.
	\end{equation}
	Since $h(0) \equiv 0 $ and $g_k$ is non-negative and increasing on $(b,\infty)$, it follows that $\tilde{\psi}(h)/\psi(h)$ is non-decreasing as well. 
	At the same time, by~\cite[Corollary~1.7]{Herbst1991} (and a standard approximation argument of $g_k\big(s\big) \1_{\{s \leq b\}}$ using $C^\infty$ functions) we have that $\ol{\rmP}_l^{\psi}$ satisfies the FKG property. It follows that 
	$\ol{\rmP}^{\tilde{\psi}}_l$ stochastically dominates $\ol{\rmP}^{\psi}_l$ and so we might as well prove the upper bound with $\ol{\rmP}^{\tilde{\psi}}_l$ in place of 
	$\ol{\rmP}^{\psi}_l$. To this end, let $u \geq 0$ and write 
	\begin{equation}
		\label{e:4.35a}
		\ol{\rmP}_l^{\tilde{\psi}} \big(h(l) > u\big) = 
		\frac{ \ol{\rmE}_l \big(\tilde{\psi}(h) ;\; h(l) > u \big)}
		{ \ol{\rmE}_l \big(\tilde{\psi}(h) \big)}.
	\end{equation}
	
	Conditioning on the value of $\min_{k \leq l} h(k)$, we may upper bound the numerator by
	\begin{multline}
		\int_{u' = u}^\infty 
		\ol{\rmP}_l \big( h(l) \in \rmd u' \big)
		\bigg(
		\ol{\rmP}_l \big(\min_{k \leq l} h(k) \geq -(b+1) \, \big|\,  h(l) = u' \big) \\
		+\sum_{v = \lceil b \rceil}^\infty \rme^{-av}  \ol{\rmP}_l \big(\min_{k \leq l} h(k) \in [-(v+1), -v] \,\big|\,  h(l) = u' \big)
		\bigg) \rme^{-a u'  1_{\{u' \geq b\}}} \,,
	\end{multline}
	where the first exponential is due to the contribution of the minimal $h$ to $\tilde{\psi}$ and the last one is due to the contribution of $h(l)$.
	
	Since $(h(k))_{k=1}^l$ is a random walk with centered Gaussian steps, standard Ballot-Type estimates (see, for example,~\cite{CHL17Supp}) give 
	\begin{equation}
		\ol{\rmP}_l \big(\min_{k \leq l} h(k) \geq -(v+1) \,\big|\,  h(l) = u' \big) \leq C\frac{v(u'+v)}{l} \leq C\frac{v^2 u'}{l} \,. 
	\end{equation}
	for some $C < \infty$ universal. Then the integral can be upper bounded by 
	\begin{equation}
		C l^{-1}  \ol{\rmE}_l \big( \rme^{-a' h(l)} ;\; h(l) \geq u \big ) 
		= C' l^{-3/2} \rme^{-a'u} \int_{s \geq 0} \rme^{-a's - \frac{(s+u)^2}{2l}} \rmd s
		\leq C'' l^{-{3/2}} \rme^{-a'u} \,,
	\end{equation}
	for some $a' < a$ and $C, C', C'' < \infty$ all depending only on $a$ and $b$. 
	
	At the same time the denominator in~\eqref{e:4.35a} only decreases if we restrict the mean to the event $\big \{\min_{k \in [1,l-1]} h(k) > b \,,\, h(l) < b \big \}$. The probability of such an event is lower bounded by 
	\begin{equation}
		\ol{\rmP}_l \big(h(l) < b\big) \, \ol{\rmP}_l \Big(\min_{k \in [1,l-1]} h(k) \geq b \, \Big|\, h(1) =0 ,\, h(l) =0 \Big)
		\geq C l^{-1/2} l^{-1} \,,
	\end{equation}
	where we have again used standard Ballot estimates. But on such an event, the value of $\tilde{\psi}(h)$ is at least $\rme^{-D}$. This shows that the denominator is at least $C l^{-3/2}$ for some $C$ which depends on $a$,$b$ and $D$ only. Combined with the upper bound on the numerator, this shows the desired statement.
\end{proof}

\subsection{Localization via a drift}
\label{ss:3.2}
For what follows let $Y=(Y_k)_{k \geq 0}$ be a non-homogenous Markov process on $\bbR$ with transition kernel 
\begin{equation}
p_{k,v}(u) := \rmP \big(Y_{k+1} \in \rmd u \,\big|\, Y_k = v \big) = \rme^{\varphi_{k,v}(u)}\rmd u \,.
\end{equation}
For $w_0 \in \bbR$, we shall write $\rmP_{w_0}$ for the law of $Y$ when $Y_0 = w_0$ and omit the subscript if it is immaterial for the statement. To control the drift of $Y$, for $a \in (0,1)$ and $w \in \bbR$,  we let $l_{a,w}: \bbR \to \bbR$ be the piecewise linear decreasing function
\begin{equation}
	l_{a,w}(u) := 2a(u-w)^- - 2a^{-1}(u-w)^+ \,.
\end{equation}
We shall assume that there exist $d,D,b \in (0,\infty)$ and $a \in (0,1)$, such that under $\rmP(\cdot \,|\,Y_k = v)$, 
\begin{equation}
\label{e:4.20}
l_{a,\ul{d}_v}(u) \leq \varphi'_{k,v}(v+u) \leq l_{a^{-1},\ol{d}_v}(u)
\end{equation}
where $\ul{d}_v \leq \ol{d}_v$ for $v \in \bbR$ are real numbers satisfying:
\begin{description}
	\item[B1] $\ul{d}_v \geq d$ for all $v < -b$
	\item[B2] $\ol{d}_v \leq -d$ for all $v > b$.
	\item[B3] $|\ul{d}_v| \vee |\ol{d}_v| < |v| + D$ for all $v \in \bbR$.
\end{description}

The key proposition here is the following. Unless stated otherwise, all constants in the remaining of this subsection may (and will) depend on $a,D,d, b$, in addition to any other explicitly specified quantity.
\begin{proposition}
\label{p:4.2}
There exist $C,c \in (0,\infty)$ such that if $ad^2 > C$, then for all $n \geq 0$, $v,v' \in \bbR$ 
\begin{equation}
	\big\| \rmP_v(Y_n \in \cdot) - \rmP_{v'}(Y_n \in \cdot) \big\|_{\rm TV} \leq C\rme^{-(cn-|v|\vee |v'|)^+} \,.
\end{equation}
\end{proposition}	

The proof of Proposition~\ref{p:4.2} will follow from the following two lemmas. For the first of the two, we 
let $Y$,$Y'$ be two independent copies of the Markov process described above with the same transition kernel. We shall write $\rmP_{v,v'}$ for $v,v' \in \bbR$ for the law under which these processes are defined and satisfy $Y_0=v$ and $Y'_0=v'$.
 Also, for $w > 0$ we let
\begin{equation}
\tau_w := \inf \{k \geq 0 :\: |Y_k| \vee |Y'_k| \leq w\} \,.
\end{equation}
\begin{lemma}
\label{l:4.5}
There exist $C, c \in (0,\infty)$ such that if $w > C$, $ad^2 > C$, then for all $v,v' \in \bbR$ \,,
\begin{equation}
\label{e:4.43a}
	\rmP_{v,v'} \big(\tau_w > k) \leq C \rme^{-(ck-|v|\vee |v'|)^+} \,.
\end{equation}
The same holds if $Y$ and $Y'$ evolve dependently only in the first step.
\end{lemma}	
\begin{lemma}
\label{l:4.7}
Let $Y$ be as above. Then for all $w > 0$ large enough there exists $\epsilon > 0$ such that for all $v,v'$ with $|v| \vee |v'| \leq w$ and all $k \geq 0$,
\begin{equation}
\label{e:4.60}
	\big\|\rmP(Y_{k+1} \in \cdot\,|\, Y_k = v) - 
		\rmP(Y_{k+1} \in \cdot\,|\, Y_k = v')\big\|_{\rm TV} < 1-\epsilon \,.
\end{equation}
In particular for all such $v, v'$, there exists a coupling $\wh{\rmP}_{v,v'}$ of $Y_{k+1}$, $Y'_{k+1}$ such that they have laws $\rmP(Y_{k+1} \in \cdot|Y_k=v)$ and $\rmP(Y_{k+1} \in \cdot|Y_k=v)$ respectively and  such that
\begin{equation}
\wh{\rmP}_{v,v'} \big(Y_{k+1} = Y'_{k+1} \big) > \epsilon \,.
\end{equation}
\end{lemma}
Let us first prove the proposition.
\begin{proof}[Proof of Proposition~\ref{p:4.2}]
Fix some $w > 0$ to be chosen later. We introduce the following coupling between two copies $Y$, $Y'$ of $Y$. Initially $Y_0 = v$ and $Y_0' = v'$ and we set $\tau_0 = -1$. For each round $m = 0, 1, 2, \dots $, if $Y_{\tau_m+1} = Y'_{\tau_m+1},$ we draw the two processes so that they stay equal from time $k=\tau_m+1$ on. Otherwise, we evolve the processes independently until the first time $\tau_{m+1}$ where
\begin{equation}
\tau_{m+1} := \inf \big\{k > \tau_m :\: |Y_k| \vee |Y'_k| \leq w \big\} \,.
\end{equation}
We now draw $Y_{\tau_{m+1}+1}$ and $Y'_{\tau_{m+1}+1}$ according to $\wh{\rmP}_{Y_{\tau_{m+1}},Y'_{\tau_{m+1}}}$ from Lemma~\ref{l:4.7}, set $m \leftarrow m+1$ and continue to the next round.

Let $M$ be the first $m \geq 0$ such that $Y_{\tau_m+1} = Y'_{\tau_m+1}$ (or infinity if no such first time exists). Set $\tau := \tau_M+1$ and observe that by construction $Y_k = Y'_k$ for all $k \geq \tau$. Now thanks to Lemmas~\ref{l:4.5} and~\ref{l:4.7}, if $w$ is large enough then
\begin{equation}
	\tau = 1 + \sum_{m=1}^M \big(\tau_m - \tau_{m-1}\big) \leq_s C(|v|\vee |v|' + 1) + \frac{1}{c} \sum_{k=1}^{\cG} \cE_k
\end{equation}
where $\cG$ is Geometric with parameter $\epsilon > 0$, $(\cE_k)_{k=1, \dots}$ are i.i.d 
with an exponentially decaying right tail, and $C,c \in (0,\infty)$ are constants -- all of which depend on $w$. We stress that $\cG$ and $(\cE_k)_{k \geq 0}$ are dependent. 
Then, with $n' := n - C(|v|\wedge |v|' + 1)$ we have
\begin{equation}
\begin{split}
\Big\|\rmP_v \big(Y_n \in \cdot)  - \rmP_{v'} & \big(Y'_n \in \cdot)  \Big\|_{\rm TV} 
\leq 
\rmP_{v,v'} \big(Y_n \neq Y'_n) \leq
\rmP_{v,v'} \big(\tau > n) 
\leq \rmP \Big(\sum_{k=1}^{\cG} \cE_k > c n'\Big) \\
& \leq \rmP \big(\cG > cn'/2\big) + \rmP \Big(\sum_{k=1}^{\lceil cn'/2 \rceil} \cE_k > c n'\Big)
\leq \rme^{-c\epsilon n'/2} + C\rme^{-c' n'/2} 
\leq C\rme^{-c''n'}
\end{split}
\end{equation}
as desired.
\end{proof}

Turning to the proof of the lemmas, we shall need the following two standard results. 
For what follows for $a,b > 0$, we denote by $N_{a,b}$ a random variable whose law is given by
\begin{equation}
	N_{a,b} \eqd (2a)^{-1} (N^+ + b) \,,
\end{equation}
where $N$ is a standard Gaussian. In particular, there exists $C >0$ such that for all $t\geq 0,$ $\rmP(N_{a,b} > t) \leq C \rme^{-a((t-b)^+)^2}.$ We shall write $N_{a}$ as a short for $N_{a,0}$.
\begin{lemma}
\label{l:103.5}
For all $b > 0$, $k \geq 1$, there exists $b' > 0$ such that for all $a > 0$ and all $t \geq 0$, 
\begin{equation}\label{eq:104.20}
	1 \wedge \big(k \rmP(N_{a,b} > t)\big)  \leq \rmP(N_{a,b'} > t) \,.
\end{equation}
\end{lemma}
\begin{proof}
For any $a,b>0$, $t\geq 0$, $\{N_{a,b}>t\}=\{N^+ -b>2at\}$, and thus we may assume without loss of generality that $a=1$ (as we are still free to choose $t\geq 0$).
Given $b$, $k$ as in the conditions of the lemma, let
\begin{equation}
	t_0:= \inf \{ t\geq 0:\, \rmP\big(N^+ +b>2t\big)\leq 1/k \} \,,
\end{equation}
and note that in particular $2t_0-b\geq 0$. Let $c>0$ and set
\begin{equation}
	b^\prime := 2 t_0+c.
\end{equation}
With this choice of parameters \eqref{eq:104.20} holds trivially for all $0\leq t\leq t_0$ as the right-hand side equals $\rmP(N^+>2t-b^\prime)=1.$
For $t>t_0$, we wish to prove
\begin{equation}\label{eq:104.24}
	\frac{\rmP \big(N^+>2t-b\big)}{\rmP \big(N^+>2t-b^\prime\big)}\leq 1/k.
\end{equation}
Now,
\begin{equation}
	\rmP \big(N^+>2t-b\big)= \rmP \big(N^+>2t_0-b\big)\frac{\rmP \big(N^+>2t-b\big)}{\rmP \big(N^+>2t_0-b\big)} 
\end{equation}
and similarly for $b^\prime$.
Thus, the left-hand side in \eqref{eq:104.24} is equal to
\begin{equation}\label{eq:104.26}
	\frac{\rmP\big(N^+>2t_0-b\big)}{\rmP \big(N^+>2t_0-b^\prime\big)} \frac{\rmP\big(N^+>2t-b\big)}{\rmP\big(N^+>2t-b^\prime\big)}\frac{\rmP\big(N^+>2t_0-b^\prime\big)}{\rmP\big(N^+>2t_0-b\big)}
\end{equation}
By the choice of $t_0$ and $b^\prime$,
\begin{equation}
	\frac{\rmP\big(N^+>2t_0-b\big)}{\rmP \big(N^+>2t_0-b^\prime\big)} \leq \frac{1}{k},
\end{equation}
\begin{equation}
	\frac{\rmP\big(N^+>2t-b\big)}{\rmP\big(N^+>2t-b^\prime\big)}=\rmP\big(N^+>2t-b\big| N^+>2t-b^\prime\big)\leq 1,
\end{equation}
and
\begin{equation}
	\frac{\rmP\big(N^+>2t_0-b^\prime\big)}{\rmP\big(N^+>2t_0-b\big)}=\rmP\big(N^+>2t_0-b^\prime\big| N^+>2t_0-b \big)=1.
\end{equation}
Plugging these back into~\eqref{eq:104.26} shows \eqref{eq:104.24} and concludes the proof.
	
\end{proof}

For the second result, let $W = (W_k)_{k \geq 0}$ be a non-negative process which is adapted to the filtration $(\cF_k)_{k \geq 0}$. We shall write $\rmP$ for the underlying probability measure and add the subscript $w \in \bbR$ to denote the case when $W_0 = w$. Suppose that for some $a,d, B > 0$ we have
\begin{equation}
\label{e:4.32}
	\big(W_{k+1} - W_k\big) \,\big|\, \cF_k \,\leq_s -d + N_a
	\quad \text{on} \quad \{W_k \geq B\} \,,
\end{equation}
and for $w > 0$, let
\begin{equation}
\label{e:103.23}
	T_w := \inf \{ k \geq 0 :\: W_k \leq w \} \,.
\end{equation}
Then,
\begin{lemma}
\label{l:4.3}
Suppose that $ad^2 > \log 2$. Then for all $w > B$, there exist $C,c \in (0,\infty)$ which may (and will) depend only on $w, a,d,B$ such that for all $z$ and all $k > 0$,
\begin{equation}
	\rmP_z \big(T_w > k) \leq C \rme^{-(ck-z)^+} .
\end{equation}
\end{lemma}
\begin{proof}
Let $W' = (W'_k)_{k=0}^\infty$ be a random walk with steps whose law is equal to that of $-d + N_a$ and with $W'_0 = z$. Denoting the probability measure on the space on which this process is defined by $\rmP'_z$. it follows from~\eqref{e:4.32}, that we may couple $W$ under $\rmP_z$ and $W'$ under $\rmP'_z$ such that $W'_l \geq W_l$ for all $l \leq T_B$. It follows that whenever $w > B$, we have
\begin{equation}
	\rmP_z (T_w > k) \leq \rmP'_z (T'_w > k) \,,
\end{equation}
where $T'$ is defined as in~\eqref{e:103.23} only with $W'_k$ in place of $W_k$.
At the same time
\begin{equation}
	\rmP'_z (T'_w > k) \leq \rmP'_z (W'_k > w) \leq 
	\rmP_z \Bigg(\sum_{l=1}^k \big(W'_l-W'_{l-1}+d\big) > w-z+kd\Bigg) .
\end{equation}
Multiplying both sides in the last probability by $\lambda = 2da$, exponentiating and using Markov's inequality, the last probability is at most
\begin{equation}
	\exp \Big(-k\Big(\lambda\big(k^{-1}(w-z)+d) - (4a)^{-1}\lambda^2 - \log 2 \Big)\Big)
	\leq \exp \Big(-k(ad^2 - \log 2) + 2ad(z - w)\Big) \,,
\end{equation}
where we used that for any $\lambda \in \bbR$, 
\begin{equation}
	\rmE \rme^{\lambda N_a} \leq 2\rme^{(4a)^{-1}\lambda^2}	 \,.
\end{equation}
This shows the desired statement.
\end{proof}

The next lemma shows that the process $|Y| = (|Y_k|)_{k \geq 1}$ satisfies the conditions imposed on $W$ above.
\begin{lemma}
\label{l:4.7a}
Let $Y$ be as above. There exist $D', B' \in (0,\infty)$ and a universal $b > 0$, such that for all $v \in \bbR$, under $\rmP(\cdot|Y_k = v)$, 
\begin{equation}
\label{e:4.51}
|Y_{k+1}| - |v| \leq_s d_v + N_{a,b} \,,
\end{equation}
where
\begin{equation}
\label{e:4.52}
d_v := -d1_{|v| > B'} + D'1_{|v| \leq B'} \,.
\end{equation}
\end{lemma}
\begin{proof}
By~\eqref{e:4.20} and FKG (see Remark~\ref{r:2.2}), for each $v$, we may couple $Z_a$, $Z'_a$ and $Y_{k+1}$ such that $Y_{k+1}$ has the same law as that under $\rmP(\cdot|Y_k = v)$ and $Z_a$ and $Z'_a$ are two random variables with laws $\rmP(Z_a \in \rmd u) = \rmP(Z'_a \in \rmd u) = C \rme^{-a^{-1}{u^-}^2-a {u^+}^2} \rmd u$ and such that
\begin{equation}
\label{e:4.38a}
v + \ul{d}_v - Z'_a \leq Y_{k+1} \leq v + \ol{d}_v + Z_a 
\end{equation}
holds almost-surely. It is easy to verify that $\rmP(Z_a > t) \leq \rmP\big(|N| > \sqrt{2a}\, t\big)$ for all $t > 0$, with $N$ a standard Gaussian, so that by the Union Bound and Lemma~\ref{l:103.5},
\begin{equation}
\label{e:103.35}
Z_A^+ \leq_s (2a)^{-1/2} |N| \leq_s N_{a,b'}\,,
\end{equation}
for some universal $b' > 0$.

Now if $v > B' := b \vee (D+d)$, then by the assumptions \textrm{B1}-\textrm{B3} on $\ol{d}_v$, $\ul{d}_v$,
\begin{equation}
-D - Z'_a \leq Y_{k+1} \leq v - d + Z_a 
\end{equation}
so that,
\begin{equation}
Y_{k+1}^+ \leq v - d + Z^+_a 
\quad, \qquad
Y_{k+1}^- \leq D + Z'^+_a \leq v - d + Z'^+_a   \,,
\end{equation}
and
\begin{equation}
	|Y_{k+1}| = Y_{k+1}^- \vee Y_{k+1}^+ \leq v-d + Z_a^+ \vee Z_a'^+ \,.
\end{equation}
It follows from~\eqref{e:103.35}, the Union Bound and Lemma~\ref{l:103.5} again that
\begin{equation}
|Y_{k+1}| - |v| \leq_s -d + N_{a,b}
\end{equation}
for some universal $b > 0$. A similar argument shows the same for $v < -B'$.

If $|v| \leq B'$, it follows from~\eqref{e:4.38a} that
\begin{equation}
-2v - D - Z'_a \leq Y_{k+1} \leq 2v + D +  Z_a 
\end{equation}
so that proceeding as before we get
\begin{equation}
|Y_{k+1}| - |v| \leq_s |v| + D + (2a)^{-1/2}(N^+ + C) \leq_s B + D + N_{a,b} \,,
\end{equation}
for $b$, $d$, $B$ as above and $D' := B+D$.
\end{proof}

We can now give
\begin{proof}[Proof of Lemma~\ref{l:4.5}]
For all $k \geq 0$, let $W_k := |Y_k| \vee |Y'_k|$. We wish to show that $W_k$ satisfies the conditions for Lemma~\ref{l:4.3} to hold. Indeed, conditioned on $Y_k = v$ and $Y_{k+1} = v'$, by Lemma~\ref{l:4.7a} we have
\begin{equation}
	\begin{split}
	W_{k+1} = |Y_{k+1}| \vee |Y'_{k+1}|
		& \leq_s \big(v + d_v + N_{a,b}\big) \vee \big( v' + d_{v'} + N'_{a,b}\big) \\
		& \leq_s (v+d_v) \vee (v'+d_{v'}) + N_{a,b} \vee N_{a,b}'
\end{split}
\end{equation}
where $N_{a,b}'$ has the same law as $N_{a,b}$ and is independent of it.

If $v \vee v' > B'+D'+d$ then the first maximum on the right hand side above is at most 
$v \vee v' - d$. Otherwise, it is at most $v \vee v' + D'$. Using also the Union Bound and Lemma~\ref{l:103.5} to handle the second maximum, we this get
\begin{equation}\label{eq:3.44}
	W_{k+1} - W_k \leq_s - d 1_{\{v > B\}} + D' 1_{\{v \leq B\}} + N_{a,b'}
\end{equation}	
with $B := B'+D'+d$ and some universal $b' > 0$. The first statement of the lemma then follows immediately from Lemma~\ref{l:4.3}, once $ad^2$ is large enough so that 
$a(d-b')^2 > \log 2$.

For the second statement, let $\rmE'$ be the expectation functional associated with the given coupling of $Y$ and $Y'$. Then, conditioning on $Y_1, Y'_1$, the left hand side in~\eqref{e:4.43a} is equal to
\begin{equation}
	\rmE'_{v,v'} \big(\rmP_{Y_1, Y'_1} \big(\tau_w > k-1\big)\big)
	\leq C\rme^{-(c(k-1) - |v|\vee |v'|)^+} \rmE'_{v,v'} \rme^{(|Y_1|-|v|)^+ \vee (|Y'_1|-|v'|)^+} \,,
\end{equation}
where we used the first part of the lemma. 
Then using~\eqref{e:4.51} and~\eqref{e:4.52}, the Union Bound and Lemma~\ref{l:103.5} the last mean is at most
\begin{equation}
C \rmE \rme^{N_{a,b'}} \,,
\end{equation}
for some $C < \infty$ and universal $b$. The result follows as the Gaussian tails of $N_{a,b'}$ ensure that the last mean is finite.
\end{proof}

To prove Lemma~\ref{l:4.7} we shall need the following lower bounds on the density of  a random variable whose log density has derivative which is bounded from above and from below by $l_{a,w}$.
\begin{lemma}
\label{l:4.6}
Let $X$ have law $\rmP(X \in \rmd u) = \rme^{\varphi(u)} \rmd u$, where $\varphi: \bbR \to \bbR$ satisfies
\begin{equation}
\label{e:4.56}
l_{a, \ul{w}}(u) \leq 
	\varphi'(u) \leq l_{a^{-1}, \ol{w}}(u) \,,
\end{equation}
for $\ul{w} \leq \ol{w}$ and $a \in (0,1)$. 
Then,
\begin{equation}
	\frac{\rmP(X \in \rmd u)}{\rmd u} \geq \frac{1}{\sqrt{\pi a}} 	 \rme^{-a^{-1} \Delta_w^2} \,
	\rme^{-\frac{(u-u_w)^2}{a}} \,,
\end{equation}
where $u_w = (\ol{w} + \ul{w})/2$ and $\Delta_w := \ol{w} - \ul{w}$.
\end{lemma}
\begin{proof}
Let $F(u) := \varphi(u) - \varphi(u_w) = \int_{u_w}^u \varphi'(u') \rmd u'$. Since
\begin{equation}
1 =  \int \rme^{\varphi(u)} \rmd u = \rme^{\varphi(u_w)} \int \rme^{F(u)} \rmd u \,,
\end{equation}
we must have 
\begin{equation}
	\rme^{\varphi(u)} = \rme^{\varphi(u_w)} \rme^{F(u)} =
	\frac{\rme^{F(u)}}{\int \rme^{F(u)} \rmd u} \,.
\end{equation}
We therefore need upper and lower bounds on $F(u)$. Suppose first that $u > u_w$. Then, by~\eqref{e:4.56},
\begin{equation}
\begin{split}
	F(u) & \geq \int_{u_w}^u l_{a,\ul{w}}(u') \rmd u'
	\geq \int_{u_w}^u (-2)a^{-1}(u'-\ul{w}) \rmd u'
	= -a^{-1}\Big((u-\ul{w})^2 - (u_w - \ul{w})^2 \Big) \\
	& \geq -a^{-1}(u-u_w)^2 - \frac12 a^{-1}\Delta_w^2\,.
\end{split}
\end{equation}
At the same time also,
\begin{equation}
\begin{split}
	F(u) & \leq \int_{u_w}^u l_{a^{-1},\ol{w}}(u') \rmd u'
	\leq \int_{u_w}^u \big(2a^{-1} \big(\ol{w} - u_w) - 2a^{-1}(u' - u_w)\big) \rmd u' \\
	& \leq -a^{-1} (u-u_w)^2 + a^{-1} \Delta_w(u-u_w) 
	\leq - a^{-1} (u-u_w)^2 + \frac12 a^{-1} \Delta_w^2 \,.
\end{split}
\end{equation}
A symmetric argument shows that the same bounds also hold when $u < u_w$.
The result follows since the upper bound on $F(u)$ implies that,
\begin{equation}
	\int \rme^{F(u)} \rmd u = \int \rme^{F(u+u_w)} \rmd u
	\leq \sqrt{\pi}\sqrt{a}\, \rme^{\frac12 a^{-1} \Delta_w^2} \,.
\end{equation}
\end{proof}

We are now ready for
\begin{proof}[Proof of Lemma~\ref{l:4.7}]
The left hand side of~\eqref{e:4.60} is equal to
\begin{equation}
	1 - \int \big(p_{k,v}(u) \wedge p_{k,v'}(u)\big)\rmd u  \,.
\end{equation}
It follows from~\eqref{e:4.20} and the assumptions on $\ul{d}_v$ and $\ol{d}_v$ that for all $w > 0$ there exist $C > 0$ such that 
$|\ul{d}_v| \vee |\ol{d}_v| \leq C$ whenever $|v| < w$. Then by Lemma~\ref{l:4.6}, there exist further $c, C \in (0,\infty)$ such that 
\begin{equation}
	p_{k,v}(u) \wedge p_{k,{v'}}(u) \geq c \rme^{-Cu^2} \,.
\end{equation}
\end{proof}

\section{Sharp control of the right tail of the minimum}
\label{s:5}
In this section we prove Theorem~\ref{t:1.1}. This will give us sharp estimates on the right tail of the minimum and its derivative. These estimates, particularly the one involving the derivative, will be crucially used in the sequel to derive the required structural results for the BRW when conditioned on $\Omega_n^+$. The proof relies on to key ingredients, which are given in the next two subsections.

\subsection{Existence of a free energy}
As a first ingredient in the proof of Theorem~\ref{t:1.1}, we shall need the following result, which shows the existence of a ``free energy'' for a specific family of Hamiltonians. This will be used to show that the error term in the theorem is $o(u)$.
\begin{proposition}
\label{l:3.2}
Let $g$, $(g_n)_{n \geq 1}$ be functions from $\bbR$ to $\bbR$ such that
\begin{description}
    \item[C1] $g_n(u) \to g(u)$ as $n \to \infty$ uniformly on compact subsets of $\bbR$.
    \item[C2] $g$ is continuous.
    \item[C3] $\lim_{u \to \pm \infty} \sup_n g_n(u) = -\infty$.
\end{description}
Then there exists $G^* \in (-\infty, \infty)$ which depends only on $g$, such that
\begin{equation}
\label{e:4.19}
\lim_{n \to \infty} \frac{1}{2^n} \log \Big[\rmE \exp \Big(\sum_{x \in \bbL_n} g_n \big(h(x)\big) \Big)\Big] = G^* \,.
\end{equation}

Moreover, let $\{(g^{(\iota)}, (g^{(\iota)}_n)_{n \geq 1}) :\: \iota \in \cI\}$ be a family of pairs of a function $g$ and a sequence $(g_n)_{n \geq 1}$ as above. Suppose that the convergence in \textrm{(C1)} is uniform also with respect to $\iota \in \cI$, that the family of functions $\{g^{(\iota)} :\: \iota \in \cI\}$ is equicontinuous on compacts and that the limit in \textrm{(C3)} is uniform with respect to $\iota \in \cI$. Then~\eqref{e:4.19} holds uniformly in $\iota \in \cI$ with $G^* = (G^{(\iota)})^*$ and $|(G^{(\iota)})^*|$ are uniformly bounded for all $\iota$.
\end{proposition}

\begin{proof}
For $n,k \geq 1$ and $u \in \bbR$, let 
\begin{equation}
\label{e:3.4}
G_{k,n} (u) := \frac{1}{2^k} \log \Big[\rmE \exp \Big(\sum_{x \in \bbL_k} g_n\big(h(x) + u\big) \Big)\Big] \ , 
\quad
G_{k} (u) := \frac{1}{2^k} \log \Big[\rmE \exp \Big(\sum_{x \in \bbL_k} g\big(h(x) + u\big) \Big)\Big] \,.
\end{equation}

The assumptions guarantee that $\lim_{u \to \pm \infty} g(u) = -\infty$ so that 
$\sup_{u\in \bbR} g(u) < \infty$ and thus there exists $M < \infty$ which bounds from above $g$ and $g_n$ for all $n$ large enough and hence also $G_k$ and $G_{k,n}$ for all $n$ large enough. From the bounded convergence theorem, we get the continuity of $G_k(u)$ and that $G_k(u)$ tends to $-\infty$ when $ u\to \pm \infty$. It follows that $G_k$ is bounded from above and attains its maximum. We therefore set $G_k^* := \max_{u \in \bbR} G_k(u)$ and let $u_k^*$ be such that $G_k^* = G_k(u_k^*)$. 

As an a priori uniform lower bound on $G_k^*$, we observe that for $u=0$, there exists $m > -\infty$ such that
\begin{equation}
G_k^* \geq G_k(0) \geq \min_{u\in[0,1]} g(u) + \frac{1}{2^k} \log \rmP \big(h(x) \in [0,1],\, \forall x \in \bbL_k \big) \geq m \,.
\end{equation}
To lower bound the probability above, we replace the event therein by that in which $h(x) \in [0,1]$ for all $x \in \bbT_k$. It is not difficult to see that this probability decays exponentially in the size of the tree with a rate that does not depend on $k$, e.g. by requiring all $2(2^k-1)$ increments to be such that for all generations $m=1,\dotsc, k-1$, starting from the first generation, $h(x) \in [-1,2]$ for all $x\in \bbL_m$.

We first claim that for all $k$,
\begin{equation}
\label{e:3.5}
	\lim_{n \to \infty} \sup_u G_{k,n}(u) = G_k^* \,.
\end{equation}
Indeed, define the event 
\begin{equation}
	A_{k,r}(u) := \Big\{ h(x) + u \in [-r, r] \,,\,\, \forall x \in \bbL_k \Big\}.
\end{equation}
Then, $\big|\rme^{2^k G_{k,n}(u)} - \rme^{2^k G_{k}(u)}\big|$ is at most 
\begin{multline}
	\rmE \bigg( \Big| \exp \Big(\sum_{x \in \bbL_k} g_n\big(h(x) + u\big)\Big) - \exp \Big(\sum_{x \in \bbL_k} g\big(h(x) + u\big) \Big) \Big| \,;\,\, A_{k,r}(u) \bigg) \\
	+
	\rmE \bigg( \exp \Big(\sum_{x \in \bbL_k} g_n\big(h(x) + u\big) \Big) \,;\, A^\rmc_{k,r}(u) \bigg)
	+ \rmE \bigg( \exp \Big(\sum_{x \in \bbL_k} g\big(h(x) + u\big) \Big) \,;\, A^\rmc_{k,r}(u) \bigg) .
\end{multline}		
The second expectation	 is bounded from above by
\begin{equation}
\exp \Big(2^k  \textstyle\sup_{|s| > r} g_n(s) \Big) \,,
\end{equation}
and thus tends to $0$ as $r \to \infty$ uniformly in $n$ and $u$. 
The same arguments, with $g_n$ replaced by $g$, shows that the last term goes to $0$ as $r \to \infty$.
At the same time, the uniform convergence of $g_n$ to $g$ on $[-r, r]$ and the boundedness of $g$ and $g_n$ again, ensure that the first term tends to $0$ as $n \to \infty$ uniformly in $u$ for all $r$. All together the above shows that
\begin{equation}
\label{e:3.9}
\lim_{n \to \infty} \sup_{u\in \bbR} \big| \rme^{2^k G_{k,n}(u)} - \rme^{2^k G_{k}(u)}\big| = 0 \,.
\end{equation}
In particular, 
\begin{equation}
\lim_{n \to \infty} \sup_{u\in \bbR} \rme^{2^k G_{k,n}(u)} = \rme^{2^k G_k^*} \,,
\end{equation}
which gives~\eqref{e:3.5}.

Next, we claim that for all $\delta > 0$,
\begin{equation}
\label{e:3.10}
\liminf_{k \to \infty} \liminf_{n \to \infty}
\min_{|u'-u_k^*| < \delta} G_{k,n}(u') > G_k^*-5\delta \,.
\end{equation}
\
To see this, we condition on the first increments, $Z$, which are standard Gaussian to write $\rme^{2^{k-1} G_k(u+\delta)}$ as
\begin{multline}
\rmE \rme^{2^{k-1} G_{k-1}(u-(Z-\delta))} 
=
\rme^{-\delta^2/2} \rmE \rme^{2^{k-1} G_{k-1}(u-Z) - \delta Z} 
\geq 
\rme^{-\delta^2/2 - (2^k) \delta} \rmE \big[\rme^{2^{k-1} G_{k-1}(u-Z)} ; |Z| < 2^k \big] \\
\geq 
\rme^{-\delta^2/2 - (2^k) \delta} \rmE \big[\rme^{2^{k-1} G_{k-1}(u-Z)} \big] - C \rme^{2^{k-1}M - 2^{2k-1}}
\geq 
\rme^{-2^{k+1} \delta} \rme^{2^{k-1} G_k(u)} - \rme^{-3^k} \, .
\end{multline}
This implies 
\begin{equation}
G_k(u+\delta) \geq G_k(u) - 4 \delta 
- \rme^{-3^k + 2^{k} (G^-_k(u))} \,,
\end{equation}
with $G^-_k(u)=\max(-G_k(u),0)$.
Repeating the derivation with $-\delta$ in place of $\delta$, gives
\begin{equation}
\min_{|u' - u| < \delta} G_k(u') \geq G_k(u) - 4 \delta 
- \rme^{-3^k + 2^{k} (G^-_k(u))} \,.
\end{equation}
Taking $u = u^*_k$, using the uniform lower bound on $G_k(u)$ and 
recalling~\eqref{e:3.9} shows~\eqref{e:3.10}.

Next, we shall show that for any $u$,
\begin{equation}
\label{e:3.6}
\lim_{k \to \infty}\lim_{n \to \infty}
\big|G_{n,n}(u) - G^*_k \big| = 0 \,.
\end{equation}
This then implies that both sequences $(G_{n,n}(u))_{n \geq 1}$ and $(G_k^*)_{k\geq 1}$ are Cauchy and therefore tend to the same finite limit which we denote by $G^*$. Plugging in $u=0$ we obtain the desired result.

We prove~\eqref{e:3.6} by upper and lower bounding the difference inside the absolute value separately. For an upper bound, by conditioning on the value of $h$ on $\bbL_{n-k}$, we get
\begin{equation}
\begin{split}
G_{n,n}(u) & = \frac{1}{2^n} \log \rmE \bigg[\exp \Big(\sum_{x \in \bbL_{n-k}} 2^k G_{k,n}\big(h(x)+u\big) \Big)\bigg]
\leq 
\frac{1}{2^n} \log \rmE \bigg[\exp \Big(2^{n-k} 2^k \sup_{u'\in \bbR} G_{k,n}(u') \Big)\bigg] \\
& \leq \sup_{u' \in \bbR} G_{k,n}(u') .
\end{split}
\end{equation}
For any $k$, the last supremum tends to $G_k^*$ as $n \to \infty$, in light of~\eqref{e:3.5}.

For a matching lower bound, let $\delta > 0$ and for any $n \geq k \geq 1$, define the event
\begin{equation}
	B_{k,n,\delta}(u) := \Big\{ \big|h(x) + u - u_k^*\big| < \delta \,,\,\, \forall x \in \bbL_{n-k} \Big\}.
\end{equation}
Conditioning again on $\bbL_{n-k}$, we then use~\eqref{e:3.10} to 
 lower bound $G_{n,n}(u)$ by
\begin{equation}
\label{e:3.19}
\frac{1}{2^n} \log \rmE \bigg[\exp \Big(\sum_{x \in \bbL_{n-k}} 2^k G_{k,n}\big(h(x)+u\big) \Big) \,;\,\, B_{k,n,\delta}(u) \bigg] 
\geq G_k^*-5 \delta + \frac{1}{2^n} \log \rmP \big(B_{k,n,\delta}(u)\big) \,,
\end{equation}
whenever $k$ and then $n$ are chosen large enough.

To lower bound the probability of $B_{k,n,\delta}(u)$, we assume without loss of generality that $u-u_k^*$ is non-negative and condition first on the event
\begin{equation}
\Big\{ h(x) \in j \wedge (u-u^*_k) + [-1,1] \,,\,\, \forall x \in \bbL_j \,,\, j = 1, \dots, n-k-1 \Big\} \,.
\end{equation}
It is not difficult to see that whenever $n-k > u-u_k^*-1$, the probability of the last event is at least $\rme^{-C|\bbT_{n-k-1}|} \geq \rme^{-C 2^{n-k}} $ for some constant $C>0$, as then each step size cannot exceed $3$ and there are $|\bbT_{n-k-1}|$ many steps. At the same time, on this event, the probability of $B_{k,n,\delta}(u)$ is at least $\rme^{-C_\delta |\bbL_{n-k}|} = 
	\rme^{-C'_\delta 2^{n-k}}$, with $C'_\delta = - \log \rmP\big(\cN(0,1) \in [1,1+2\delta]\big)$. Altogether we get for all $n$ large enough,
\begin{equation}
\frac{1}{2^n} \log \rmP \big(B_{k,n,\delta}(u)\big) 
\geq -2^{-k} (C+C'_\delta) \,.
\end{equation}
For any $\delta > 0$, the right hand side is at least $-\delta$ once $k$ is chosen large enough. 

Altogether this shows that for all $\delta > 0$, the right hand side in~\eqref{e:3.19} is at least $G_k^* - 6\delta$ provided first $k$ and then $n$ are chosen sufficiently large. In particular, this shows the desired lower bound in~\eqref{e:3.6}.

The second part of the Theorem, which is that the statement \eqref{e:4.19} in case of a family of functions $\{(g^{(\iota)}, (g^{(\iota)}_n)_{n \geq 1}) :\: \iota \in \cI\}$ holds uniformly in $\cI$, follows from the fact that by assumption all bounds above can be performed uniformly in $\iota \in \cI$.
Moreover, by Assumption {\rm (C3)} it follows first that $(G^{\iota})^*$ are bounded uniformly from above as {\rm (C3)} holds also uniformly in $\iota \in \cI$. At the same time the uniformity of the Assumption {\rm (C3)} implies that $u^*_k(\iota)$ are bounded from below and above uniformly in both $k$ and $\iota$, and thus by the equicontinuouity of $ \{(g^{(\iota)}:\: \iota \in \cI\}$, $(G^{\iota})^*$ are also bounded from below uniformly in $\iota \in \cI$.
\end{proof}

\subsection{Control of the conditional mean}
The second ingredient in the proof of Theorem~\ref{t:1.1} is the following control over the mean value of $h$ conditional on the event $\Omega_n(u)$. The more difficult proof of the upper bound uses the representation of $\rmP_n(-|\Omega_n(u))$ as a random walk under a pinning potential. Here we appeal to the the localization statements from Subsection~\ref{ss:3.1}, which are in force, thanks to the a priori estimates on the right tail of the minimum, as stated in Subsection~\ref{s:2.5.2}.
\begin{proposition}
\label{p:3.1}
For all $\epsilon > 0$ there exists $C > 0$ such that such that for all $n \geq 1$, $1 \leq u \leq 2^{\sqrt{n}}$, $x \in \bbL_{l_u}$, with $l_u = \lfloor \log_2 u \rfloor$,
\begin{equation}
	\Big| \rmE_n \big(h(x) \,\big|\, \Omega_n(u) \big) - (u - c_0 \log_2 u)\Big| \leq C \,.
\end{equation}
\end{proposition}

\begin{proof}
In view of Proposition~\ref{l:2.8a} we can assume that $u > u_0$ for some arbitrary but fixed $u_0$.
Starting with the lower bound on the mean, using representation~\eqref{e:2.32d}, the conditional mean in the statement of the proposition can be written as $\rmE_l^{\varphi} h(x)$, where $\varphi$ is equal to $\varphi_{n,l,u}$ as defined in~\eqref{e:2.40d}. We wish to claim that the conditional mean only decreases if we consider $\rmE_l^{\psi} h(x)$ instead, with the function $\psi$ defined as
\begin{equation}
\psi(h) \equiv \psi_{n,l,u}(h) := \prod_{x \in \bbL_l} \tilde{p} \big(-h(x) + u'\big) \ , \quad \tilde{p}(s) := \rme^{-\frac18 (s-C_0)^2} \,,
\end{equation}
for some sufficiently large $C_0 \in (0,\infty)$, $s\in \bbR$ and with
\begin{equation}
u' = u - (m_n - m_{n-l}) = u - c_0 l + O(l/n) \,.	
\end{equation}

Indeed, since $\rmP_l^+$ is FKG by Lemma~\ref{lemma:FKG}, it is sufficient to show that $\psi(h)/\varphi(h)$ is decreasing in $h$. Checking one coordinate at a time, taking the logarithm and differentiating, this would be the case if for all $l \geq 1$ and all $s \in \bbR$,
\begin{equation}
	\frac{\rmd}{\rmd s} \log  p_l(s) \leq \frac{\rmd}{\rmd s} \log  \tilde{p}(s) =
	-\frac14 (s - C_0) \,.
\end{equation}
But this follows for $C_0$ large enough from the upper bound in Lemma~\ref{l:4.2a}.

Now, let $\wt{\bbT}_{l+1}$ be the tree $\bbT_{l+1}$ with one child of each vertex at depth $l$ removed and denote by $\wt{\bbL}_{l+1}$ its leaves. Consider a DGFF on $\wt{\bbT}_{l+1}$ where all conductances are $1$ except on the edges leading to vertices in $\wt{\bbL}_{l+1}$ where they are equal to $1/4$, and boundary values $\psi :\: \{0\}\cup \wt{\bbL}_{l+1} \to \bbR$, where $\psi(0) = 0$ and $\psi(z) = u' - C_0$ for $z \in \wt{\bbL}_{l+1}$. Denoting the law of such field by $\wt{\rmP}_{l+1}$ we thus have,
\begin{equation}
	\rmP_l^{\psi}(-) = \wt{\rmP}_{l+1} (-) \,.
\end{equation}
It then follows from~\eqref{e:2.7} that 
\begin{equation}
	\rmE_n \big(h(x) \,\big|\, \Omega_n(u)\big) \geq \ol{\psi}(x) \,,
	\end{equation}
where $x \in \bbL_l$. As in~\eqref{e:2.8b} the right hand side above is equal to $(u'-C_0)$ times the probability that a random walk on $\wt{\bbT}_{l+1}$ (with conductances as above) starting from $x$ reaches $\wt{\bbL}_{l+1}$ before reaching the root. A similar Gambler-Ruin-type calculation as in~\eqref{e:2.9b}, shows that this probability is at least $1-C/2^l$, so that
\begin{equation}
\ol{\psi}(x) \geq (u'-C_0) (1-C2^{-l})-C' \geq u' - C_0 - C2^{-l} u' - C' \,.
\end{equation}
Plugging in $l = l_u$, we obtain the desired lower bound.

Turning to the upper bound, recalling~\eqref{e:2.43a}, the law of $h$ on the branch $[x]_l$ for $k \leq l$ and $x \in \bbL_l$ can be written as
\begin{equation}
\label{e:4.21}
\rmP_n \Big(\big(h([x]_k)\big)_{k=0}^l \in \cdot \,\Big|\, \Omega_n(u) \Big) =	\ol{\rmP}_l^{\varphi} (\cdot) \,,
\end{equation}
where $\varphi$ is equal to $\varphi_{n,l,u}$ from~\eqref{e:2.38a} (with $l$ in place of $k$).
As in the proof for the lower bound, we start by replacing $\ol{\rmP}_k^{\varphi}$ with a more explicit measure $\ol{\rmP}_l^{\psi}$, where $\psi$ is now defined as
\begin{equation}
	\label{e:4.4}
	\psi(h) \equiv \psi_{n,l,u}(h) :=  \prod_{k=1}^l \tilde{p}_{n-k}(-h(k) + u - m_n + m_{n-k})^{1+\1_l(k)}
\end{equation}
and
\begin{equation}
\label{e:3.13}
	\tilde{p}_k(s) := \exp \Big(-\frac{(s^+)^2 + 2C_0 s}{4-2^{-k+2}} \Big) \,,
\end{equation}
for some constant $C_0 > 0$. The constant $C_0$ is chosen large enough so that, by Lemma~\ref{l:4.2a}, 
\begin{equation}
	\frac{\rmd}{\rmd s} \log \hat{p}_k(s)
		= \frac12 \frac{\rmd}{\rmd s} \log p_k(s) \geq 
		-\frac{1}{2-2^{-k+1}}\big( s^+  + C_0\big) = \frac{\rmd}{\rmd s} \log \tilde{p}_k(s) \,.
\end{equation}
As $\ol{\rmP}_l^\varphi$ is FKG by Lemma~\ref{lemma:FKG}, 
the mean of $h(x)$ only increases if we replace $\ol{\rmP}_l^\varphi$ by $\ol{\rmP}_l^\psi$. 

Next, we consider a tilted version of the law $\ol{\rmP}_l$. To this end, we define $\mu \equiv \mu_{n,l,u}$ via
\begin{equation}
	\mu(k) := (1-2^{-k}) (u - m_n + m_{n-k})
	\quad, \quad k \in [0,l] \,.
\end{equation}
Then, an easy calculation shows that for $k \in [1,l]$,
\begin{equation}
\label{e:3.16}
		\big(\Delta \mu\big)(k) = \big(\mu(k)-\mu(k+1)\big)\1_{k<l}+\left(\mu(k)-\mu(k-1)\right) =
	2^{-k-1+1_{\{l=k\}}} (u-m_n + m_{n-k}) + C_k\,,
\end{equation}
where $\Delta$ is the Laplacian on the linear graph $[0,l]$, and $C_k \equiv C_{n,l,u,k}$ are uniformly absolutely bounded for all $k,n,l$ and $u$.

Abbreviating $u_k \equiv u - m_n + m_{n-k} = u - c_0 k + O(k/n) $, for all $k \in [1,l]$, 
\begin{equation}
\label{e:4.16a}
\begin{split}
- \log \tilde{p}_{n-k}  \big(& -h(k)   - \mu(k) + u_k\big)^{1+\1_l(k)}  + h(k) (\Delta \mu)(k)  \\
	& = \tfrac{1+\1_l(k)}{4-2^{-n+k+2}} \Big( \big( (-h(k) + 2^{-k} u_k )^+\big)^2 - 2C_0 h(k)) \Big) + \big(2^{-k-1+\1_l(k)} u_k+C_k\big)h(k) + C  \\
	& = \tfrac{1+\1_l(k)}{4-2^{-n+k+2}} \Big( h(k)^2 - 2^{-n+1} u_k h(k) - \big((-h(k) + 2^{-k} u_k )^-\big)^2\Big) + C'_k h(k) + C' \\
	& = \tfrac{1+\1_l(k)}{4-2^{-n+k+2}} \Big( \big(h(k) - C''_k \big)^2 - \big((h(k) - 2^{-k} u_k )^+\big)^2\Big) + C'' \,,
\end{split}
\end{equation}
where all constants may depend on $n,k,l,u$ (but non $h$), and $C_k''$ are bounded in absolute value uniformly in these parameters. Above we have absorbed all terms which do not involve $h$ into the constants $C$, $C'$ and $C''$. Such constants will anyhow get cancelled out, once we consider the tilted (and hence normalized) version of $\rmP_k$. Crucial in the above computation is the cancellation of terms of the form $O(u_k h(k))$. This will be important below and is due to the careful choice of $\mu$. 

Then by Lemma~\ref{l:1}, for any $F: \bbR^{[0,l]} \to \bbR$ we have
\begin{equation}
	\label{e:4.18}
	\ol{\rmE}_l^{\psi} F(h) = \frac{\ol{\rmE}_l \big( F(h) \psi(h) \big)}{\ol{\rmE}_l \big( \psi(h) \big)} 
	= \frac{\ol{\rmE}_l \Big( F(h+\mu) \psi(h+\mu) \big) \rme^{-\langle h, \Delta \mu \rangle} \Big)}
	{\ol{\rmE}_l \Big( \psi(h+\mu) \rme^{-\langle h, \Delta \mu \rangle} \Big)} 
	= \ol{\rmE}_l^{\chi} F(h+\mu) \,,
\end{equation}
where $\chi \equiv \chi_{n,l,u} : \bbR^{[0,l]} \to \bbR$ is defined as 
\begin{equation}
\chi(h) :=  \exp \Big( -\sum_{k=1}^l g_{k} \big((h(k)\big) \Big) \,,
\end{equation}
with
\begin{equation}
	g_k(s) \equiv g_{n,k,l,u}(s) := \tfrac{1+\1_l(k)}{4-2^{-n+k+2}} \Big( \big(s - C''_k \big)^2 - \big((s - 2^{-k} u_k )^+\big)^2\Big) \,.
\end{equation}
(Notice that the constant $C''$ from~\eqref{e:4.16a} does not appear in the definition of $g_k$.)

Since all constants are bounded uniformly, there exists $C_1 < \infty$ such that with $l := l_u - C_1$ for all $k \in [1,l]$ we have $2^{-k} u_k \geq 2^{C_1}/4 > C_k''$. 
It follows that the functions $(g_k)_{k=1}^l$ satisfy the conditions of Proposition~\ref{p:2.14} with some universal $a,b,d$ and thus that under $\ol{\rmE}_l^{\chi}$ that random variable $h(l)$ has (at least) an exponentially decaying upper tail with universal constants. Altogether we obtain
\begin{equation}
\rmE_n \big(h(x_l) \,\big| \Omega_n(u) \big) 
\leq
	\ol{\rmE}_l^{\psi} \big(h(l)\big) =
	\ol{\rmE}_l^{\chi} \big(h(l) + \mu(l) \big) \leq
	u' + C\,,
\end{equation}
with $l = l_u - C_1$. By conditioning on $h(x_{l_u - C_1})$ and using Proposition~\ref{l:2.8a}, this upper bound extends to $l=l_u$ as well.
\end{proof}

\subsection{Proof of main theorem}
We can now give
\begin{proof}[Proof of Theorem~\ref{t:1.1}]
Let $u \geq 0$ be large, and, for $0 \leq k \leq n$, set $u_k = u - m_n + m_{n-k} = u - c_0 k + O(k/n)$. Recall from~\eqref{e:2.41e} that the probability in question can be recast as
\begin{equation}
\label{e:4.38}
	p_n(u) = \rmE_{k} \varphi_{n,k,u}(h)
\end{equation}
where
\begin{equation}
\varphi_{n,{k},u}(h) := \prod_{x \in \bbL_{k}} p_{n-{k}} \big(-h(x) + u_k \big) \,.
\end{equation}
Next, define $\mu : \bbT_{k} \to \bbR$ such that $\mu(0) = 0$, $\mu_{{\bbL_{k}}} = u_k$ and such that $\mu$ is discrete harmonic on $\bbT_{k} \setminus \{0\} \setminus \bbL_{k}$. As in~\eqref{e:2.8b}, we have
$\mu([x]_{{k}-1}) = u_k(1-(2^{k}-1)^{-1})$ for $x \in \bbL_{k}$. It follows that,
\begin{equation}
\begin{split}
\frac12 \langle \mu, \Delta \mu \rangle + \langle h, \Delta \mu \rangle & =
\frac12 \sum_{x \in \bbL_{k}} \mu(x) \big(\mu (x) - \mu ([x_{{k}-1}) \big) +
\sum_{x \in \bbL_{k}} h(x) \big(\mu(x) - \mu ([x]_{{k}-1}) \big) \\
& =  \frac{u_k^2}{2(1-2^{-{k}})}  + \frac{u_k}{2^{{k}}-1} \sum_{x \in \bbL_{k}} h(x)  \,.
\end{split}
\end{equation}

We may then use Lemma~\ref{l:1} to rewrite the mean in~\eqref{e:4.38} as
\begin{equation}
\label{e:3}
\rme^{-\frac{u_k^2}{2-2^{-{k}+1}}}
	\rmE_{k} \prod_{x \in \bbL_{k}} p_{n-{k}} (-h(x)) \rme^{-\frac{u_k}{2^{k} - 1} h(x)}\,,
\end{equation}
We now pick $k = {l_u} := \lfloor \log_2 u \rfloor = \log_2 u - [u]_2$, where $[u]_2$ was defined in~\eqref{e:1.7} and set 
\begin{equation}
\label{e:4.42a}
	u' := u - c_0 \lfloor \log_2 u \rfloor \,,
\end{equation}
Then $u_{l_u} = u' + O(n^{-1}\log u)	$, the first exponent in~\eqref{e:3} is equal to
\begin{equation}
\label{e:4.42}
	-\frac{u_{l_u}^2}{2}\big(1+2^{-l_u} + O\big(2^{-2l_u}\big) \big) 
	= -\frac{u'^2}{2} - 2^{[u]_2-1} u + o(u) \,,
\end{equation}
and the second exponent therein is equal to
\begin{equation}
-\big(2^{[u]_2} + f(u)\big) h(x) \,,
\end{equation}
with $f(u) = O(u^{-1}\log u)$ as $u \to \infty$.

Consider now the functions $g^{\delta, a, b},\, g_n^{\delta, a, b} :\ \bbR \to \bbR$ for $n \ge 1$, $\delta \in (0,1)$, $a \in [1,\infty)$, $b \in [-1,1]$, defined for all $s \in \bbR$, via
\begin{equation}
	g^{{\delta, a, b}}_n(s) := \log \Big(p_{\lfloor a n \rfloor}(-s) \rme^{-(2^{\delta}+b/n) s}\Big) \,,
\end{equation}
and
\begin{equation}
	g^{{\delta, a, b}}(s) := \log \Big(p_\infty(-s) \rme^{-2^{\delta} s}\Big) \,.
\end{equation}
Then the product in~\eqref{e:3} can be written as
\begin{equation}
	\exp \Big\{ \sum_{x \in \bbL_{l_u}} g^{{\delta, a, b}}_{l_u} \big(h(x)\big) \Big\} \,,
\end{equation}
where $\delta = [u]_2$ and $a$ and $b$ are chosen so that $a l_u  = n-l_u$ and $b/l_u = f(u)$. This is always possible for all $u$ large enough and $n \geq 1$, by the restriction on $u$ in the statement of the theorem. We claim that the functions in the family above satisfy the assumptions in Proposition~\ref{l:3.2} uniformly as required therein. Indeed, Assumption \textrm{(C1)} and~\textrm{(C2)} follow as argued below~\eqref{e:2.32}. Assumption \textrm{(C3)} follows thanks to Lemma~\ref{l:4.2a}, which shows that $\log p_{n}(s) \leq -(s^+)^2/2 + C$ for all $s$ and $n$. It then follows from Proposition~\ref{l:3.2} and~\eqref{e:4.42} that the logarithm of~\eqref{e:3} is equal to
\begin{equation}
	-\frac{u'^2}{2} - 2^{[u]_2-1} u + (G^{\delta,a,b})^* 2^{l_u} + o(2^{l_u})
	=  -\frac{(u - c_0 \log_2 u)^2}{2} - \theta_{[u]_2}u + e_n(u) \,,
\end{equation}
where $\theta : [0,1) \to \bbR$ is some bounded function and where $e_n(u)/u \to 0$ as $u \to \infty$ uniformly in $n$ as required.

Next, we turn to the second statement in the theorem. In view of Lemma~\ref{l:2.9a} we may assume that $u$ is large enough as needed. We now take the logarithm and differentiate~\eqref{e:3} with respect to $u$. This gives
\begin{equation}\label{eq:4.43}
\begin{split}
	\frac{\rmd}{\rmd u} \big(-\log p_n(u)\big) & = 
	\frac{u_k}{1-2^{-{k}}} + p_n(u)^{-1} \rme^{-\frac{u_k^2}{2-2^{-{k}+1}}} \frac{1}{2^{k}-1} \sum_{x \in \bbL_{k}}  \rmE_{k} \Big[ h(x) 
	\prod_{y \in \bbL_{k}} p_{n-{k}} (-h(y)) \rme^{-\frac{u_k}{2^{k} - 1} h(y)}
	\Big] \\
	& = \frac{u_k}{1-2^{-{k}}}
		+ \frac{1}{1-2^{-{k}}} \rmE_n \big(h(x) - u_k \,\big|\, 
		\Omega_n(u) \big) = \frac{1}{1-2^{-{k}}} \rmE_n \big(h(x)  \,\big|\, 
		\Omega_n(u)\big) \,.
\end{split}
\end{equation}
Above, $x \in \bbL_{k}$ and we used the symmetry of leaves to replace the sum with $2^{k}$ times one term, followed by a tilting back by $\mu$ above.
Choosing $k = l_u$ and using Proposition~\ref{p:3.1} we get
\begin{equation}\label{eq4.44}
	\frac{\rmd}{\rmd u} \big(-\log p_n(u)\big) = u' + O(1)
\end{equation}
with $u'$ as in~\eqref{e:4.42a}.
\end{proof}

\section{Reduction to a random walk subject to a localization force}
\label{s:6}
Given the sharp estimates on the right tail of the minimum, we can express the law of $h$ on a branch from the root under the conditioning on $\Omega_n(s)$ as the law of a random walk with attraction to zero, of the sort considered in Section~\ref{s:3}. 
We shall recast the law of $h$ as both a random walk subject to a pinning potential and a random walk with a localizing drift. 
Proofs in this section are deferred to the end.

We begin with the pinning potential case. Such a reduction was already done in the proof of Proposition~\ref{p:3.1}, albeit using coarser estimates for the right tail of the minimum. 
\begin{proposition}
\label{p:4.1}
There exist $C,c \in (0,\infty)$ such that the following holds. Let $n \geq C$, $u \in (C,2m_n]$, $v \in \bbR$, $x \in \bbL_n$ and set $l_u := \lfloor \log_2 u \rfloor$ and $u' := u-c_0 l_u$. Let also $l \leq l_u - C$. Then,
\begin{equation}
\label{e:4.1}
\rmP_n \Big( \Big(h\big([x]_k\big) - u'(1-2^{-k}) + v
 \: : \: k \in [0, l] \Big) \in \cdot \,\Big|\, 
\Omega_n(u-v) \Big) = \ol{\rmP}_l^\chi \big(h \in \cdot\,\big|\, h(0) = v \big) 
\end{equation}
where
\begin{equation}
\label{e:5.2b}
	\chi(h) := \exp \Big(-\sum_{k=1}^l f_k\big(h(k)\big) \Big) \,,
\end{equation}
and $f_k \equiv f_{n,l,u,k} : \bbR \to \bbR$ satisfy
\begin{equation}
\label{e:4.3}
c\big(|s| - (s - 2^{-k-1}u')^+\big) - C \, \leq \,
\sgn(s) f'_k(s) \, \leq \, C\big(|s| - (s - 2^{-k-1}u')^+\big) + C \,
\end{equation}
for all $k=1, \dots, l$ and all $s \in \bbR$.
\end{proposition}

Next we consider the localizing drift case. As the proof shows, this is in fact a consequence of the previous proposition.
\begin{proposition}
	\label{p:4.2a}
	There exists $C_0 \in (0,\infty)$ such that the following holds.
	Let $n \geq C_0$, $u \in (C_0,2m_n]$ and $x \in \bbL_n$. Set  $l_u := \lfloor \log_2 u \rfloor$, 
	$l'_u := l_u - C_0$, $u' = u - c_0 l_u$, and for all $k \in [0,l'_u]$ also 
	\begin{equation}
		Y_k := h\big([x]_{k}\big) - (1-2^{-k})u' \,.
	\end{equation}
	Then there exist $a, d, D, b \in (0,\infty)$ such that under $\rmP_n\big(-\,\big|\, \Omega_n(u)\big)$ the process $Y = (Y_k)_{k=0}^{l_u'}$ is of the type described in Subsection~\ref{ss:3.2} with $Y_0 = 0$. Moreover, $ad^2$ can be made arbitrarily large by	 choosing $C_0$ arbitrarily large.
\end{proposition}

\begin{proof}[Proof of Proposition~\ref{p:4.1}]
Suppose first that $v = 0$, so that $w=u$. Abbreviate, as before, 
\begin{equation}
u_k \equiv u - m_n + m_{n-k} = u - c_0 k + O(k/n)\,,
\end{equation}
and let 
\begin{equation}
\mu(k) := u'(1-2^{-k}) \,.	
\end{equation}
As in the proof of Proposition~\ref{p:3.1},  we use the representation in Subsection~\ref{s:2.6.2} to write the left hand side in~\eqref{e:4.1} as
\begin{equation}
	\ol{\rmP}_l^{\varphi}\big( h - \mu \in \cdot \big)
\end{equation}
where
\begin{equation}
	\varphi(h) := \exp \Big(-\sum_{k=1}^l -\log \hat{p}_{n-k} \big(u_k - h(k)\big) \big(1+\1_l(k)\big) \Big) \,.
\end{equation}
Setting for $t \in \bbR$, 
\begin{equation}
	\psi(h) := \exp \Big(-\frac{1}{4} \sum_{k=1}^l \big(t - h(k)\big)^2 \big(1+\1_l(k)\big) \Big) \,,
\end{equation}
we may write
\begin{equation}
\label{e:4.22}
	\ol{\rmP}_l^{\varphi}\big( h - \mu \in \cdot \big) = \Big(\ol{\rmP}_l^{\psi_{u'}}\Big)^{\varphi/\psi_{u'}} \big( h - \mu \in \cdot \big)\,.
\end{equation}

The advantage of the last rewrite is that $\ol{\rmP}_l^{\psi_{u'}}$ can be seen as the law of a DGFF on a graph $\bbG'$ with vertices $\bbV' = \{0,\dots, l\} \cup \{1', \dots, l'\}$ and edges $\bbE' = \big\{\{k,k-1\}, \{k, k'\} :\: k =1, \dots, l\big\}$. This graph has conductances which are $1$ on $\{k,k-1\}$ and $\frac12(1+\1_l(k))$ on $\{k, k'\}$ for all $k \in [1,l]$. Furthermore, this DGFF is taken to have boundary conditions $\varphi$ on $\{0\} \cup \{k' :\: k \in [1,l]\}$, given by $\varphi(0) = 0$ and $\varphi(k') \equiv u'$ for all $k$.

It is not difficult to check that $\mu = \ol{\varphi}$ and therefore, by~\eqref{e:2.7}, we may take boundary conditions $\varphi \equiv 0$ at the cost of adding $\mu$ the resulting field. This equates the laws in~\eqref{e:4.22} with
\begin{equation}
\label{e:4.11}
	\Big(\ol{\rmP}_l^{\psi}\Big)^{(\varphi/\psi_{u'})(\cdot + \mu)} \big(h \in \cdot \big)
	= \ol{\rmP}_l^{\xi} \big(h \in \cdot\big)
\end{equation}
where
\begin{equation}
\xi(h) = \psi_0(h) \psi(h+\mu)/ \psi_{u'}(h+\mu) =
\exp \Big(-\sum_{k=1}^l f_k \big(h(k)\big) \big(1+\1_l(k)\big) \Big) \,,
\end{equation}
\begin{equation}
\label{e:4.12}
f_k \big(u\big) = \tfrac14 u^2  + g_k(u)
\end{equation}
and
\begin{equation}
\label{e:4.13}
g_k(s) := -\log \hat{p}_{n-k} \big(u_k - (s+\mu(k)) \big) - 
	\frac{1}{4} \big(u' - (s+\mu(k))\big)^2 \,.
\end{equation}
Observe that 
\begin{equation}
\label{e:105.14}
s' - \mu(k) = 2^{-k} s' 
\quad , \qquad 
	u_k - \mu(k) = 2^{-k} u' + c_0 (\lfloor \log_2 u\rfloor - k) + O(k/n) \,,
\end{equation}

To show~\eqref{e:4.3}, whenever $s \leq 2^{-k-1} u'$, we have
\begin{equation}
	u_k - \mu(k) - s \geq 2^{-k-1} u' - 1 \,,
\end{equation}
so that by the Mean Value Theorem, 
\begin{equation}
\frac12 c_0 \Big| \log_2 \big(u_k - (s - \mu(k)\big) - \big(\lfloor \log_2 u\rfloor  - k\big) \Big|
	\leq \frac12 c_0 \Big(1+\frac{|s|}{2^{-k-1} u'-1} \Big) 
	\leq \frac14 |s| + C \,,
\end{equation}
assuming that $k \leq l \leq l_s - C'$ for $C'$ large enough and also that $u$ is large enough.

In this range of $s$ we can use Theorem~\ref{t:1.1} and Relation~\eqref{e:102.45} to claim that
\begin{equation}
	\bigg|\frac{\rmd}{\rmd s} g_k(s) \bigg| \leq \frac14 |s| + C \,.
\end{equation}

On the other hand, if $s > 2^{-k-1} u'$, we lower bound the derivative of the first term in~\eqref{e:4.13} by $0$ and upper bound it by $C < \infty$, using Lemma~\ref{l:2.9a}, so that
\begin{equation}
	\bigg|\frac{\rmd}{\rmd s} g_k(s) - \frac{1}{2} \big(2^{-k}u' - s \big)\bigg| \leq  C \,.
\end{equation}
Combining these two bounds and using that the derivative of the first term of $f_k(s)$ is $\frac12 s$, we thus get~\eqref{e:4.3}.

Now if $v \neq 0$, we write the left hand side of~\eqref{e:4.1} as 
\begin{equation}
\rmP_n \Big( \Big(h\big([x]_k\big) - u'(1-2^{-k}) - v
 \: : \: k \in [0, l] \Big) \in \cdot \,\Big|\, 
\Omega_n(u),\, h(0) = v \Big) \,,
\end{equation}
with the formal conditioning used to specify boundary conditioning at the root. We now proceed exactly as before. In particular, we use that $\ol{\rmP}_l^{\psi_{u'}}(\cdot\,|\,h(0) = v)$ is now the law of a DGFF on $\bbG'$ with $\varphi$ equal to $v$ at $0$ and $u'$ at $k'$ for all $k \in [1,l]$. Again, we replace the boundary conditions $u'$ with that of $0$, keeping the value $v$ as the height of the field at $0$, at the cost of adding $\mu(k)$ to $h(k)$ for all $k \in [1,l]$. All together this equates the probability in question with
\begin{equation}
	\ol{\rmP}_l^{\xi} \big(h \in \cdot \,\big|\, h(0) = v \big)
\end{equation} 
where $\xi$ is defined exactly as before.
\end{proof}

\begin{proof}[Proof of Proposition~\ref{p:4.2a}]
The Markovian property of the process $Y$ follows from Gibbs-Markov structure of $h$. 
Now, by definition, for any $k \in [0,l'_u]$, $s,t \in \bbR$, the density 
\begin{equation}
\rme^{\varphi_{t,k}(s)} := \frac{\rmP(Y_{k+1} \in \rmd s | Y_k = t)}{\rmd s} 
\end{equation}
is equal to 
\begin{equation}
\label{e:105.22}
\begin{split}
\rmP_n& \Big(h\big([x]_{k+1}\big) - (1-2^{-k-1})u'  \in \rmd s \,\Big|\, 
		h\big([x]_{k}\big)  = t + (1-2^{-k})u' \,,\,\, \Omega_{n}(u)\Big)/\rmd s \\
& =
\rmP_{n-k}\Big(h\big([x']_1\big) - 2^{-k-1}u'\ + t  \in \rmd s \,\Big|\, 
\Omega_{n-k}(w-t)\Big)/\rmd s \\
& = 
\rmP_n\Big(h\big([x']_1\big) - w'/2 + t  \in \rmd s' \,\Big|\, 
\Omega_{n-k}(w-t) \Big) /\rmd s'
\,,
\end{split}
\end{equation}
where $x' \in \bbL_{n-k}$, 
\begin{equation}
w := u - (1-2^{-k})u' - m_n + m_{n-k} = 2^{-k}u' + c_0 l_u - c_0 k + O(k/n)\,,
\end{equation}
\begin{equation}
w' := w - c_0 l_w
\ , \quad l_w := \lfloor \log_2 w \rfloor \,,
\end{equation}
and
\begin{equation}
\label{e:5.25}
	s' = s - (w'/2 - 2^{-k-1} u') = s - O\Big(\log \big(2^{-k} u\big)/ \big(2^{-k} u\big) + k/n\Big) \,,
\end{equation}
with the $O$ terms made arbitrarily small by choosing $C_0$ large enough.

Thanks to Proposition~\ref{p:4.1}, we therefore have
\begin{equation}
\varphi_{t,k}(s) = -\tfrac12(s'-t)^2 - f_1(s') + C\,,
\end{equation}
so that,
\begin{equation}
	\varphi'_{k,t}(s) = 
	-(s'-t) - f'_1(s') \,,	
\end{equation}
It then follows from~\eqref{e:4.3} that 
\begin{equation}
\ul{l}(s') \leq \varphi'_{k,t}(s) \leq
\ol{l}(s')
\end{equation}
where 
\begin{equation}
\ul{l}(s') := -(1+c) s' 1_{(-\infty, 0]}(s') - (1+C) s' 1_{(0, w'/4]}(s') - 
(C w'/4 + s') 1_{(w'/4, \infty)}(s') + 
t - C  
\end{equation}
and
\begin{equation}
\ol{l}(s') := -(1+C) s' 1_{(-\infty, 0]}(s') - (1+c) s' 1_{(0, w'/4]}(s') - 
(c w'/4 + s') 1_{(w'/4, \infty)}(s') + t + C \,.
\end{equation}
Thanks to the bounded derivative of $\ul{l}$ and $\ol{l}$ and~\eqref{e:5.25}, by potentially increasing $C$, we can ensure that also
\begin{equation}
\ul{l}(s) \leq 	\varphi'_{k,t}(s) \leq 
\ol{l}(s) \,.
\end{equation}

Now, both $\ol{l}$ and $\ul{l}$ are piecewise linear functions which are decreasing and have minimal and maximal slopes $-(1+C)$ and $-1$ respectively. Letting $\ul{s}$, $\ol{s}$ be the unique values of $s$ such that $\ul{l}(\ul{s}) = \ol{l}(\ol{s}) =0$, it follows that both
\begin{equation}
	\ul{l}(s) \geq (s-\ul{s})^- -(1+C)(s-\ul{s})^+
\end{equation}
and
\begin{equation}
	\ol{l}(s) \leq (1+C)(s-\ol{s})^- -(s-\ol{s})^+\,.
\end{equation}
Moreover, it is not difficult to see that 
\begin{equation}
\ul{s} \geq -\frac{(t-C)^-}{1+c} + \frac{(t-C)^+}{1+C}
\end{equation}
and
\begin{equation}
\ol{s} \leq -\frac{(t+C)^-}{1+C} + \frac{(t+C)^+ \wedge (1+c)w'/4}{1+c} + 
	\big((t+C)^+ - (1+c)w'/4\big)^+ \,.
\end{equation}

Setting $\ul{d}_t := \ul{s}-t$, $\ol{d}_t := \ol{s}-t$ and plugging in $s+t$ in place of $s$, we thus get
\begin{equation}
	(s-\ul{d}_t)^- -(1+C)(s-\ul{d}_t)^+
	\leq \varphi'_{k,t}(t+s) \leq
		(1+C)(s-\ol{d}_t)^- -(s-\ol{d}_t)^+\,,
\end{equation}
with all constants universal, as long as $C_0$ is chosen large enough. Setting $a := (1+C)^{-1}$ we thus get~\eqref{e:4.20} using Remark~\ref{r:2.2}.

It is easy to verify that $\ul{d}_t$, $\ol{d}_t$ satisfy both
\begin{equation}
	\ul{d}_t \geq 
	\Big(-\frac{c}{1+c}t + \frac{C}{1+c}\Big) 1_{\{t \leq 0\}} + \
	\Big(-\frac{C}{1+C}t - \frac{C}{1+c}\Big) 1_{\{t > 0\}} 
\end{equation}
and 
\begin{multline}
	\ol{d}_t \leq
	\Big(-\frac{C}{1+C}t + \frac{C}{1+c}\Big) 1_{\{t \leq 0\}} + \
	\Big(-\frac{c}{1+c}t + \frac{C}{1+c}\Big) 1_{\{t \in [0, (1+c)w'/4-C]\}} \\
 	+\Big(C-cv'/4\Big) 1_{\{t \in [(1+c)w'/4-C, \infty)\}} \,.
\end{multline}

For any $d > 0$, by taking $C_0$ large enough in the statement of the proposition, we may ensure that $w'/4$ is so large, that there exists $w_0 > 0$ for which $\ol{d}_t \leq -d$ whenever $t > w_0$ and $\ul{d}_t > d$ whenever $t < -w_0$. Since also $\ol{d}_t \vee \ul{d}_t \leq |t| + D$ for all $t$ and some $D > 0$, 	the conditions in Subsection~\ref{ss:3.2} are satisfied. Moreover, since $a$ does not depend on $C_0$, and $d$ can be made arbitrarily large by choosing $C_0$ properly, the same applies to $ad^2$.
\end{proof}

\section{Mean, fluctuations and covariances}
\label{s:7}
In this section we prove Theorems~\ref{t:1.3}, Corollary~\ref{c:1.4} and Theorem~\ref{t:1.4}.  For what follow, we abbreviate,
\begin{equation}
	\label{e:102.1}
	\hat{h}(x) \equiv h(x) - \mu_n(x) \,,
\end{equation}
with $\mu_n$ as in~\eqref{e:1.10a}. 

\begin{proof}[Proof of Theorem~\ref{t:1.3}]
Let $u > 0$, $x \in \bbT_n$, $k := |x|$ and suppose first that 
$k \leq l'_n := l_n - C_0$, for some $C_0$ large enough to be determined later.  Set
\begin{equation}
\ul{k}_u := \big(\big\lfloor l_n - \log_2 (u\vee 1) \big\rfloor\big) \wedge k \vee 0
\quad , \qquad
\ol{k}_u := \tfrac12 \big(k+\ul{k}_u\big) \,.
\end{equation}
Thanks to Proposition~\ref{p:4.1} with $r=0$ and $s = m_n$, the law $\big(h\big([x]_i\big) - m_{n'}(1-2^{-i}) :\: i \leq k\big)$ is $\ol{\rmP}_l^\chi$, namely a random walk $(h(i) :\: i \leq k)$ with standard Gaussian steps, starting from $0$ and measure changed by
\begin{equation}
	\chi(h) := \exp \Big(-\sum_{i=1}^k f_i\big(h(i)\big) \Big) \,,
\end{equation}
with $f_k$ satisfying~\eqref{e:4.3}. 

By choosing $f_i(0)$ large enough for all $i \leq k$, as we are free to do, and also 
taking $C_0$ large, we can guarantee, via~\eqref{e:4.3}, that the function $v \mapsto \sgn(v)f_i(v)$ is increasing whenever $|v| > C$, and that for all $v \in \bbR$, 
\begin{equation}
	f_i(v) \in \big(c(|v|(v^-+1)\,,\,\, C(|v|^2 + 1)\big) ,\, \forall i \in [0, k] \,,
\end{equation}
and 
\begin{equation}
	f_{\ul{k}_u}(v) > c v (v \wedge u) 
	\ , \quad 
	f_{i}(v) < C\big(v(v \wedge u 2^{-\frac12(k-\ul{k}_u)}
	2^{-(i-\ol{k}_u)}+1)\big) ,\, \forall i \in \big[\ol{k}_u, k\big] \,.
\end{equation}
	
Given $l \in [1, k)$, we now split $\chi$ into $\chi'_l$ and $\chi''_l$ where
\begin{equation}
	\chi'(h) := \exp \Big( - \tfrac12 f_l\big(h(l) \big)-\sum_{i=1}^{l-1} f_i\big(h(i)\big)\Big) \,,
\end{equation}
and
\begin{equation}
	\chi''(h) := \exp \Big(- \tfrac12 f_l\big(h(l)\big) -\sum_{i=l+1}^k f_i\big(h(i)\big)\Big)  \,,
\end{equation}
so that $\chi(h) = \chi'(h) \chi''(h)$. In particular, conditioning on $h(i)$ for $i \leq l$,  we may write the probability in~\eqref{e:1.9a} as 
\begin{equation}
\label{e:4.43}
	\frac{\ol{\rmE}_{l}^{\chi'} \Big[ \ol{\rmE}_k \big(\chi''(h) ; \; h(k) > u  \, \big|\, h(l) = v \big)\big|_{v=h(l)} \Big]}
{\ol{\rmE}_{l}^{\chi'} \Big[ \ol{\rmE}_k \big(\chi''(h) \, \big|\, h(l) = v\big)\big|_{v=h(l)} \Big]}.
\end{equation}

Now, for an upper bound on the upper tail, we choose $l = \ul{k}_u$ in the above decomposition. Then for $A > 0$ large enough, the denominator is at least
\begin{equation}
\label{e:4.46}
\ol{\rmP}_{\ul{k}_u}^{\chi'} \big( |h(\ul{k}_u)| < A \big)
c\rme^{-C(k-\ul{k}_u)} \ol{\rmP}_{k-\ul{k}_u} \Big( \max_{0 \leq i \leq k-\ul{k}_u} |h(i)| < A \Big) \,,
\end{equation}
where $C := \max_{i \leq k} \|f_i\|_{\bbL^\infty([-2A,2A])}$ depends on $A$, but uniformly bounded in $n$. Thanks to Proposition~\ref{p:2.14}, for $A$ large enough we may bound from below the first probability in the right hand side by a positive constant. By a standard random walk estimate, the last probability is at least $c'\rme^{-C'(k-\ul{k}_u)}$ for some $C',c' \in (0, \infty)$ which depend on $A$. The denominator in~\eqref{e:4.43} can thus be lower bounded by $c''\rme^{-C''(k-\ul{k}_u)}$.

At the same time, the conditional mean in the numerator is at most 
\begin{equation}
\label{e:4.27}
	\rme^{- cv (v \wedge u)/2}
	\ol{\rmP}_{k-\ul{k}_u} \big( h(k-\ul{k}_u) > u-v \big)
\leq \exp \Big(- c'v (v \wedge u) -c\frac{\big((u-v)^+\big)^2}{\ol{\sigma}_n(x,u)}\Big) \,,
\end{equation}
where the first exponent accounts for the contribution of $\tfrac12 f_{\ul{k}_u}(h(\ul{k}_u))$ to the sum in the exponent and the second is due to the Gaussian tail of $h(k-\ul{k}_u)$, which under $\ol{\rmP}_{k-\ul{k}_u}$ has variance $k-\ul{k}_u \leq \big(\log_2 (u \vee 1) - 
(l_n - k)\big)^+ + C \leq C\ol{\sigma}_n(x,u)$.

If $v < u/2$, the second exponent is at most $-cu^2/\ol{\sigma}_n(x,u)$. In the case when $v > u/2$, we upper bound the the first by $-c'u^2/4 < -cu^2/\ol{\sigma}_n(x,u)$. 
Combined with the lower bound on the denominator of~\eqref{e:4.43}, this gives the desired upper bound on the upper tail.

For a matching lower bound, we take $l = \ol{k}_u$ in~\eqref{e:4.43}. We then upper bound the denominator trivially by $1$, as all $f_i$-s are non negative. The numerator on the other hand is at least
\begin{multline}
\label{e:5.44}
\ol{\rmP}_{\ol{k}_u}^{\chi'} \big( |h(\ol{k}_u)| < A \big) 
\exp \Bigg(-C u^2 2^{-\frac12(k-\ul{k}_u)} \sum_{i = 0}^{k-\ol{k}_u} 2^{-i} \Bigg) \\
\rmP_{k-\ol{k}_u} \Big(h(k-\ol{k}_u) \geq u + A \,,\,\,
 |h(i)| \in [-A, 2u + A] :\: i \in [0, k-\ol{k}_u]
  \Big)\,,
\end{multline}
where $C$ depends on $A$. Again by Proposition~\ref{p:2.14}, the first probability above is at least $c > 0$ for $A$ large enough, and the second is at least
\begin{equation}
c'\rme^{-C'u^2/(k-\ol{k}_u)} 
= c'\rme^{-2C'u^2/(k-\ul{k}_u)} \geq c'\rme^{-C''u^2/\ol{\sigma}_n(x,u)} \,,
\end{equation}
for some $C',c' \in (0,\infty)$ which depend on $A$, by a standard random walk estimate. As the sum in the exponent in~\eqref{e:5.44} is uniformly bounded, and $2^{1/2(k-\ul{k}_u)} > c(k-\ul{k}_u + 1) > c'\ol{\sigma}_n(x,u)$, combined with the upper bound on the denominator, we get the desired lower bound on the upper tail.

For the lower tail the argument is similar and even simpler. We take $l = k$ in~\eqref{e:4.43}. A lower bound on the denominator is obtained exactly as in the case of the upper bound, and now amounts to $c' \rme^{-C'(k-l)} > c > 0$.
The numerator is at most $\sup \big\{\rme^{-f_k(-v)/2} :\: v > u\big\} \leq C\rme^{-cu^2}$.
For the opposite bound, we upper bound the denominator by $1$ and lower bound the numerator by
\begin{equation}
\ol{\rmP}_{k}^{\chi'} \big( |h(k)| < A \big)	
c\rme^{-Cu^2}
\ol{\rmP}_{1} \big( h(1) \in [-u-1, -u] \big)	
\geq 
c'\rme^{-C' u^2} \,,
\end{equation}
where we used that $h(1)$ is a standard Gaussian under $\ol{\rmP}_1$ and that $f_k(v) < C(|v|^2+1)$.

We turn to the case when $k > l'_n$. Conditioning on $h(y)$ with $y=[x]_{l'_n}$, using the just established bounds, FKG, Lemma~\ref{l:2.4} and since $\inf_{v \geq 0} \rmP_{n-l_n'}(\Omega_{n-l_n'}(-v)) > 0$, by the tightness of the minimum, we may upper bound the upper tail by
\begin{equation}
\label{e:5.52a}
\begin{split}	
 \rmP_{n-l'_n} & \big(h(x') > u \,\big|\, \Omega_{n-l_n'}(0)\big)
	+
	\sum_{v = 1}^\infty 
		\rmP_n^+ \big(\hat{h}(y) \geq v-1\big)		\rmP_{n-l_n'} \Big(h(x') > u-v \,\big|\, \Omega_{n-l_n'}(-v)\Big) \\
& \leq C\rme^{-u^2/|x'|}
	+ C \int_{v = 1}^\infty \rme^{-cv^2/(\log_2 v + 1)}  \rme^{-(u-v)^2/2|x'|} \rmd v
\leq C' \rme^{-c'u^2/(\log_2 (u \vee 1) + |x'|)} \,, 
\end{split}
\end{equation}
where $x' \in \bbL_{k - l'_n}$ and where the last inequality follows from a standard computation.
For the other direction, we use FKG to lower bound the desired probability by
\begin{equation}
\rmP_n^+ \bigg(\hat{h}(y) \geq \frac{\ol{\sigma}_n(y,u)/2C}{\ol{\sigma}_n(y,u)/2C + |x'|}u \bigg)
	\rmP_{n-l_n'}  \bigg(h(x') \geq \frac{|x'|}{\ol{\sigma}_n(y,u)/2C + |x'|} u\bigg)
		 \,,
\end{equation}
Then, by the just established lower bound on the upper tail and since $h(x')$ is a centered Gaussian with variance $|x'|$, the latter is at least
\begin{equation}
	c \exp \bigg(-\frac{u^2}{\ol{\sigma}_n(y,u)/2C + |x'|} \bigg)
	\geq c'\exp \bigg(-C'\frac{u^2}{\ol{\sigma}_n(x,u)} \bigg) \,.
\end{equation}
Lastly we treat the lower tail when $k > l_n'$. For an upper bound, we have as before,
\begin{equation}
\begin{split}	
 \rmP_{n-l'_n} & \big(h(x') < -u \,\big|\, \Omega_{n-l'_n}(0)\big)
	+
	\sum_{v = 1}^\infty 
		\rmP_n^+ \big(\hat{h}(y) \leq -v+1\big)
		\rmP_{n-l'_n} \Big(h(x') < -u+v \,\big|\, \Omega_{n-l'_n}(v)\Big) \\
& \leq C\rme^{-u^2/|x'|}
	+ C \int_{v = 1}^\infty \rme^{-cv^2} \rme^{-(u-v)^2/2|x'|} 
\leq C \rme^{-c'u^2/|x'|} \,, 
\end{split}
\end{equation}

For a matching lower bound, by Lemma~\ref{l:2.4}, we lower bound the desired probability by
\begin{equation}
\label{e:105.56}
P_n^+ \big(\hat{h}(y) \leq 0\big)
P_{n-l_n'} \big(h([x]) \leq -u \,|\, \Omega_{n-l_n'}(0)\big)
\, \geq \,
c P_{n-l_n'} \big(h([x]) \leq -u \,,\,\, \Omega_{n-l_n'}(0)\big)\,,
\end{equation}
where in the second inequality we used the already established lower bound on the lower tail.
At the same time, the last probability in~\eqref{e:105.56} is at least
\begin{multline}
\label{e:105.58}
	P_{n-l'_n} \Big(h(x') \in [-u-1, -u] \Big) \\
	P_{n-l'_n} \bigg(\min_{1 \leq i \leq k-l_n'} \Big(h([x']_i) + m_{i} - 
	 (i \wedge (k-l'_n))^{1/4} \Big) \geq 0 \,\bigg|\, h(x') = -u\bigg) \\
	\prod_{i \leq k-l_n'} 
	\bigg(P_{n'-i} 
	\bigg(\min_{\bbL_{n'-i}} h \geq -m_{n'} + m_{i}  -
	 (i \wedge (k-l'_n-i))^{1/4} \bigg)\bigg)^{\frac12(1+\1_{\{i=k-l_n'\}})}\,.
\end{multline}

The first term in the product above is lower bounded by $\rme^{-Cu^2/|x'|}$ by a standard Gaussian tail bound.
The second probability can be written as 
\begin{equation}
\label{e:105.59}
	P_{n-l'_n} \bigg(\min_{1 \leq i \leq k-l_n'} \bigg(h([x']_i) + m_{i} -
	\frac{i}{k-l'_n} u - (i \wedge (k-l'_n))^{1/4} \bigg) \geq 0 \,\bigg|\, h(x') = 0\bigg) \\
\end{equation}
If $u \leq m_{k-l_n}/2$, then  
\begin{equation}
m_{i} - \frac{i}{k-l'_n} u - (i \wedge (k-l'_n))^{1/4} \geq 	ci-C \,,
\end{equation}
so that the last probability can be bounded from below by a uniform constant $c > 0$.
Otherwise, since $u \leq m_{k-l_n}$ as in the statement of the theorem, we lower bound the last left hand side by
\begin{equation}
m_{i} - \frac{i}{k-l'_n} m_{k-l_n} - (i \wedge (k-l'_n))^{1/4} \geq 	
-c(i \wedge (k-l'_n))^{1/4} - C\,.
\end{equation}
In this case, the probability in~\eqref{e:105.59} is at least $c/(|x'|+1)$ by a standard Ballot-Type estimate (c.f.,~\cite{CHL17Supp}). 

Finally, since
\begin{equation}
	-m_{n'} + m_{i}  -
	 (i \wedge (k-l'_n-i))^{1/4}
	 \leq -m_{n'-i} - c(i \wedge (k-l'_n-i))^{1/4} + C,
\end{equation}
we may use Lemma~\ref{l:2.10} to lower bound the probability in the product in~\eqref{e:105.58} by
\begin{equation}
	1 - C\rme^{-c(i \wedge (k-l'_n-i))^{1/4}} \leq 	\exp \Big(-C\rme^{-c(i \wedge (k-l'_n-i))^{1/4}}\Big) \,.
\end{equation}
Since the last exponent is summable in $i$, this shows that the product in~\eqref{e:105.58} can be lower bounded by a constant $c > 0$.

Collecting all bounds, the lower tail probability in~\eqref{e:1.9b} is lower bounded by 
\begin{equation}
	c\rme^{-Cu^2/|x'|} \Big((1+|x'|)^{-1} + \1_{\{u \leq m_{|x'|	}/2\}}\Big) 
	\geq c'\rme^{-C'u^2/|x'|} \,.
\end{equation}
\end{proof}

Combining Theorem~\ref{t:1.3} and Proposition~\ref{l:2.8a}, the proof of Corollary~\ref{c:1.4} is straightforward.
\begin{proof}[Proof of Corollary~\ref{c:1.4}]
Let $x \in \bbT_n$. When $|x| \leq l_n$, we have $\ul{\sigma}_n(x) \leq 1$ and $\ol{\sigma}_n(x) \leq \log_2 u + 1$, so that 
 $(h(x) - \mu_n(x) :\: n \geq 0, |x| \leq l_n)$ is uniformly integrable in $u$ by Theorem~\ref{t:1.3}  and hence~\eqref{eq:1.14} follows.
Turning to the case $|x| \geq l_n$, for $0 \leq k \leq n$, with $x_k \in \bbL_k,$ set
\begin{equation}
	\nu_{n,k}(u) := \rmE_{n} \big(u + h(x_k) \,\big|\, \Omega_{n}(-u)\big) \,.
\end{equation}
By Proposition~\ref{l:2.8a}, we have $u \leq \nu_{n,k}(u) \leq u + u^- + C \leq u^+ + C$.
Therefore, 
\begin{equation}
 \rmE_n^+ \big(h([x]_{l_n}) - m_{n'}\big) \leq  \rmE_n^+ \nu_{n', k-l_n} \big(h([x]_{l_n}) - m_{n'}\big) \leq \rmE_n^+ \big(h([x]_{l_n}) - m_{n'}\big)^+ + C\,.
\end{equation}
Since the first and last mean are bounded by a constant, thanks to the first part of the proof, it remains to observe that the middle expectation is precisely the left hand side of~\eqref{eq:1.14} when conditioning on $h([x]_{l_n})$ and using the Tower Property.
\end{proof}

For the covariances, we have
\begin{proof}[Proof of Theorem~\ref{t:1.4}]
We shall prove~\eqref{e:1.14} with $h$ replaced by $\hat{h}$. Let $k := |x\wedge y|$ and suppose first that $k \geq l'_n := l_n - C_0$ for some $C_0 > 0$ to be chosen later.
Then conditioning on $\hat{h}([x]_k) = \hat{h}([y]_k)$ and using the total covariance formula, symmetry and the Gibbs-Markov property of $h$, 
\begin{equation}
\label{e:6.64}
	\Cov_n^+ \big(\hat{h}(x), \hat{h}(y)\big) = 
	\Var_n^+ \big(\rmE_n^+ \big(\hat{h}(x)\,\big|\ \hat{h}([x]_k) \big)\big) \,.
\end{equation}	
Conditioning further on $\hat{h}([x]_{l'_n})$ the latter is equal to
\begin{equation}
\label{e:1.16}
	\Var_n^+ \big(\rmE_n^+\big(\hat{h}(x)\,\big|\ \hat{h}([x]_{l'_n} \big)\big) +
	\rmE_n^+ \Big(\rmVar_n^+ \Big(\rmE_n^+\big(\hat{h}(x)\,\big|\ \hat{h}([x]_{k}\big) \big)\,\big|\, \hat{h}([x]_{l'_n})\Big) \Big) \,.
\end{equation}	

We wish to show that the first term in~\eqref{e:1.16} is $O(1)$ and the second is $O(k-l'_n)$. To this end by Proposition~\ref{l:2.8a} with $x' \in \bbL_{n'}$ we have
\begin{equation}
\Big|\rmE_n^+\big(\hat{h}(x)\,\big|\ \hat{h}([x]_{l'_n}) = v \big)\Big|
= \Big|v+\rmE_{n'} \big(h(x')\,\big|\, \Omega_{n-l'_n}(-v)\big) \Big|
\leq |v| + C\,.
\end{equation}
 Thus, the first term in~\eqref{e:1.16} is at most 
\begin{equation}
	C \Big(\rmE_n^+ \big(\hat{h}([x]_{l'_n})\big)^2 + 1\Big)
	\leq C' \,,
\end{equation}
by the tails of $\hat{h}([x]_{l'_n})$ under $\rmE_n^+$, as given by Theorem~\ref{t:1.3}.

Turning to the second term in~\eqref{e:1.16}, we note that for $v \in \bbR$, 
\begin{equation}
\label{e:1.19}
\rmE_n^+\big(\hat{h}(x)\,\big|\ \hat{h}([x]_k) = v \big)
= v + \rmE_{n-k}\Big(h(x'')\,\big|\ \Omega_{n-k}\big(-v - \tilde{c}_{n,k} \big)\Big) \,,
\end{equation}
with $x'' \in \bbL_{n-k}$ and $\tilde{c}_{n,k} := \mu_n([x]_k) - m_{n-k} = m_{n-l'_n} - m_{n-k} + O(1)$.
Therefore, by Proposition~\ref{l:2.8a}, 
\begin{equation}
\hat{h}([x]_k) \leq \rmE_n^+\big(\hat{h}(x)\,\big|\ \hat{h}([x]_k) \big)
\leq \hat{h}([x]_k) + \big(\hat{h}([x]_k) + \tilde{c}_{n,k})^- + C \,.
\end{equation}

Using that if $X,Y,Z$ are random variables such that a.s. $Y \leq X \leq Y + Z$ and $Z \geq 0$, then $\rmVar\, X \leq \rmVar\, Y + \rmE Z^2 + 2\sqrt{(\rmVar\, Y)\rmE Z^2}$, we thus have
\begin{multline}
\rmVar_n^+ \Big(\rmE_n^+\big(\big(\hat{h}(x) \,\big|\ \hat{h}([x]_{k}) \big)\,\big|\, \hat{h}([x]_{l'_n})=v\Big) \\ \leq 
\rmVar_n^+ \big( \hat{h}([x]_k) \,\big|\,\hat{h}([x]_{l'_n})=v\big) 
+ C\rmE_n^+ \Big(\big(\big(\tilde{c}_{n,k} + \hat{h}([x]_k))^-\big)^2 + 1 
 \,\big|\,\hat{h}([x]_{l'_n})=v\Big)  \\
 + 
C'\sqrt{\rmVar_n^+ \big( \hat{h}([x]_k) \,\big|\,\hat{h}([x]_{l'_n})=v\big) 
\rmE_n^+ \Big(\big(\big(\tilde{c}_{n,k} + \hat{h}([x]_k))^-\big)^2 + 1 
 \,\big|\,\hat{h}([x]_{l'_n})=v\Big) } \,.
\end{multline}
Letting $x''' \in \bbL_{k - l'_n}$, for the mean above we bound
\begin{multline}
	\rmE_n^+ \Big(\big(\big(\tilde{c}_{n,k} + \hat{h}([x]_k))^-\big)^2 
 \,\big|\,\hat{h}([x]_{l'_n}) = v\Big) \leq
	\rmE_n \Big(\big(\big(\tilde{c}_{n,k} + \hat{h}([x]_k))^-\big)^2 
 \,\big|\,\hat{h}([x]_{l'_n}) = v\Big) \\
  = 	\rmE_{n - l_n'} \big(\big(\tilde{c}_{n,k} + v + h(x'''))^-\big)^2 
\leq
C \Big(\big((k-l'_n + v)^-\big)^2 + (k-l'_n) \Big) \exp \bigg(-c\frac{\big((k-l'_n + v)^+\big)^2}{k-l'_n} \bigg) \,,
\end{multline}
where we used that $x'''$ is a centered Gaussian with variance $k-l'_n$ under $\rmP_{n'}$
and that $\tilde{c}_{n,k} \asymp (k-l'_n)$. 
At the same time, by Proposition~\ref{l:2.13},
\begin{equation}
\rmVar_n^+ \big( \hat{h}([x]_k) \,\big|\,\hat{h}([x]_{l'_n}) = v\big) 
= \rmVar_{n-l_n'} \Big( h(x''') \,\big|\Omega_{n-l_n'}(-v +  O(1))\big)\Big)  =  k - l'_n + O \big(v^2 + 1\big) \,.
\end{equation}

Altogether, for the second variance in~\eqref{e:1.16} we have
\begin{multline}
\Big|\rmVar_n^+ \Big(\rmE_n^+\big(\big(\hat{h}(x) \,\big|\ \hat{h}([x]_{k}) \big)\,\big|\, \hat{h}([x]_{l'_n}) = v\Big) \Big) - (k-l'_n) \Big| \\ 
\leq C \bigg(|v|^2 + (k-l'_n)^2 \exp \bigg(-c\frac{\big((k-l'_n + v)^+\big)^2}{k-l'_n} \bigg) + 1\bigg)\,.
\end{multline}

Since $\hat{h}([x]_{l'_n})$ has almost-Gaussian tails under $\rmP_n^+$ by Theorem~\ref{t:1.3}, integrating the right hand side w.r.t $\rmP_n^+ \big(\hat{h}([x]_{l'_n} \in \rmd v\big)$ yields, 
\begin{equation}
\Big|\rmE_n^+ \Big(\rmVar_n^+ \Big(\rmE_n^+\big(\big(\hat{h}(x)\,\big|\ \hat{h}([x]_{k}) \big)\,\big|\, \hat{h}([x]_{l'_n})\Big) \Big) - (k-l'_n)\Big| \leq C < \infty \,.
\end{equation}
Collecting all estimates we get,
\begin{equation}
	\Cov_n^+ \big(\hat{h}(x), \hat{h}(y)\big) = k-l'_n + O(1) = 
		|x\wedge y| -l_n + O(1) \,.
\end{equation}

Suppose now that $k = |x \wedge y| < l'_n$. Setting
\begin{equation} 
Z_n := \rmE_n^+ \big(\hat{h}(x)\,\big|\ \hat{h}([x]_{l'_n})\big) \,,
\end{equation}
for any $M > 0$ we may use the Tower Property of covariance and expectation to write
\begin{equation}
\label{e:6.77a}
\begin{split}
	\Cov_n^+ \big(\hat{h}(x), \hat{h}(y)\big) & = \Var_n^+ \big( \rmE_n^+\big(\hat{h}(x)\,\big|\, \hat{h}([x]_k)\big)\big)  = 
	\Var_n^+ \Big(\rmE_n^+ \big(Z_n\,\big|\, \hat{h}([x]_k)\big)\Big) \\	
	& \leq 2 \Var_n^+ \Big(\rmE_n^+ \big(Z_n \1_{\{|Z_n| \leq M\}}\,\big|\, \hat{h}([x]_k)\big)\Big) + 2 \Var_n^+ \big(Z_n \1_{\{|Z_n| > M\}}\big) \,.
\end{split}
\end{equation}
Now, by Proposition~\ref{l:2.8a} we have
\begin{equation}
\hat{h}([x]_{l'_n}) - C\leq Z_n
\leq \hat{h}([x]_{l'_n})^+ + C \,,
\end{equation}
so that the last variance in~\eqref{e:6.77a} is at most
\begin{equation}
	C \rmE_n^+ \Big(\hat{h}([x]_{l'_n})^2 \1_{\{|\hat{h}([x]_{l'_n})| > M \}}\Big) \,,
\end{equation}
which is itself at most $C \rme^{-cM^{3/2}}$, thanks to the almost-Gaussian tails of $\hat{h}([x]_{l'_n})$ as stated in Theorem~\ref{t:1.3}.

At the same time, the first variance in the second line of~\eqref{e:6.77a} is at most
\begin{equation}	
	\rmE_n^+ \Big(\rmE_n^+ \big(Z_n\1_{\{|Z_n| \leq M\}}\,\big|\ \hat{h}([x]_k) \big) - \rmE_n^+ \big(Z_n\1_{\{|Z_n| \leq M\}}\,\big|\ \hat{h}([x]_k)=0 \big)\Big)^2
	\leq \rmE_n^+ \Big(M^2 \rme^{-c(l'_n - k - |\hat{h}([x]_k)|)^+} \Big) \,,
\end{equation}
where we have appealed to Proposition~\ref{p:4.2} for the process $Y_m = \hat{h}([x]_m)$ for $m \leq l'_n$, which satisfies the requirements in the proposition, thanks to Proposition~\ref{p:4.2a} with $u = m_n$ and $C_0$ chosen large enough.
Again due to the almost-Gaussian tails of $\hat{h}([x]_k)$ by Theorem~\ref{t:1.3}, the last mean is at most $C M^2 \rme^{-c(l'_n-k)}$. Combining this with the bound on the second variance in~\eqref{e:6.77a}, the desired covariance is at most
\begin{equation}
	C\big(\rme^{-cM^{3/2}} + M^2 \rme^{-c(l'_n-k)}\big) \,.
\end{equation}
Taking, e.g., $M = (l'_n-k)$ and recalling that $l'_n - k = l_n - |x \wedge y| + O(1)$, this gives the desired bound.
\end{proof}

\section*{Acknowledgments}
This research was supported through the programme "Research in Pairs"  by the Mathematisches Forschungsinstitut Oberwolfach in 2020 and also partly funded by the Deutsche Forschungsgemeinschaft (DFG, German Research Foundation) under Germany's Excellence Strategy - GZ 2047/1, Projekt-ID 390685813 and GZ 2151 - Project-ID 390873048,
through the Collaborative Research Center 1060 \emph{The Mathematics 
of Emergent Effects}. 
The research of M.F. was also partly supported by a Minerva Fellowship of the Minerva Stiftung Gesellschaft fuer die Forschung mbH. The research of L.H. was supported by the Deutsche Forschungsgemeinschaft (DFG, German Research Foundation)  through Project-ID 233630050 - TRR 146, Project-ID 443891315 within SPP 2265, and Project-ID 446173099.
The research of O.L. was supported by the ISF grant no.~2870/21, and by the BSF award 2018330. 
The authors would like to thank Amir Dembo for useful discussions.

\bibliographystyle{abbrv}
\bibliography{bbm-AL-ref.bib}
\end{document}